\documentclass{amsart}
\usepackage[margin=1in]{geometry}

\usepackage{eurosym}
\usepackage{amsmath,amssymb,amsthm}
\usepackage{amsfonts,amsrefs}
\usepackage{url}
\usepackage[pdfencoding=unicode,psdextra]{hyperref}
\usepackage{bookmark}
\usepackage{graphicx,caption}
\usepackage{bbm}

\usepackage{bbm}
\usepackage[export]{adjustbox}
\usepackage[section]{algorithm}
\usepackage[noend]{algpseudocode}

\theoremstyle{plain} 
\newtheorem{theorem}[algorithm]{Theorem}
\newtheorem{corollary}[algorithm]{Corollary}
\newtheorem{lemma}[algorithm]{Lemma}
\newtheorem{proposition}[algorithm]{Proposition}

\theoremstyle{definition}
\newtheorem{definition}[algorithm]{Definition}
\newtheorem{remark}[algorithm]{Remark}
\newtheorem{example}[algorithm]{Example}

\usepackage{chngcntr}
\counterwithin{figure}{section}

\usepackage{microtype}
\DisableLigatures{encoding = *, family = *}
\usepackage[T3,OT1]{fontenc}
\DeclareSymbolFont{tipa}{T3}{cmr}{m}{n}
\DeclareMathAccent{\invbreve}{\mathalpha}{tipa}{16}

\begin{document}

\title[Coefficients of Catalan States of Lattice Crossing I]{Coefficients of Catalan States of Lattice Crossing I: \\ $\Theta_{A}$-state Expansion}

\author{Mieczyslaw K. Dabkowski}
\address{Department of Mathematical Sciences, The University of Texas at Dallas, Richardson, TX 75080}
\email{mdab@utdallas.edu}

\author{Cheyu Wu}
\address{Department of Mathematical Sciences, The University of Texas at Dallas, Richardson, TX 75080}
\email{cheyu.wu@utdallas.edu}

\begin{abstract}
Plucking polynomial for plane rooted trees was introduced by J.H.~Przytycki in 2014. As it was shown later, this polynomial can be used to find coefficients $C(A)$ of Catalan states $C$ of $m \times n$-lattice crossing $L(m,n)$ without returns on one side. In this paper, we show that $C(A)$ for any $C$ can be found by using $\Theta_{A}$-state expansion which represents $C(A)$ as a linear combination of coefficients of Catalan states with no top returns over $\mathbb{Q}(A)$. We also provide an algorithm for finding $\Theta_{A}$-state expansions and examples of its applications. Finally, an example of a Catalan state with non-unimodal coefficient is given.
\end{abstract}

\maketitle

%\tableofcontents

\section{Introduction}
\label{s:intro}
Kauffman Bracket Skein Module (KBSM)\footnote{Skein modules were introduced by J.H.~Przytycki \cite{Prz1991} in 1987 and later discovered by V.~Turaev \cite{Tur1990} in 1988.} for a product of an oriented surface of genus $g$ with $n$ boundary components $F_{g,n}$ and an interval $I = [0,1]$ is an algebra\footnote{Multiplication of links $L_{1}$ and $L_{2}$ in KBSM of $F_{g,n} \times I$ is defined by placing $L_{1}$ above $L_{2}$.} called the Kauffman Bracket Skein Algebra (KBSA) of $F_{g,n}$. In 2000, D.~Bullock and J.H.~Przytycki studied these algebras for $(g,n) \in \{(1,0),(1,1),(1,2),(0,4)\}$ and found their presentations (see \cite{BP2000}). Furthermore, an elegant and important \emph{``product-to-sum formula''} for multiplication of curves in KBSA of $F_{1,0}$ was given by C.~Frohman and R.~Gelca in \cite{FG2000} and later used in \cites{Gel2002,GU2003,McL2007}. Unfortunately, formulas for product in KBSA of $F_{g,n}$ with $(g,n) \in \{(1,1),(1,2),(0,4)\}$ appear to be much more complex and difficult to find. Some partial results have been though obtained. For instance, formulas for product of some special curves in KBSA of $F_{1,1}$ were given in \cites{Cho2012,CG2014} and for product of some classes of curves in KBSA of $F_{0,4}$ were obtained in \cite{Li2016} and \cite{BMPSW2018}.

\begin{figure}[ht] 
\centering
\includegraphics[scale=1]{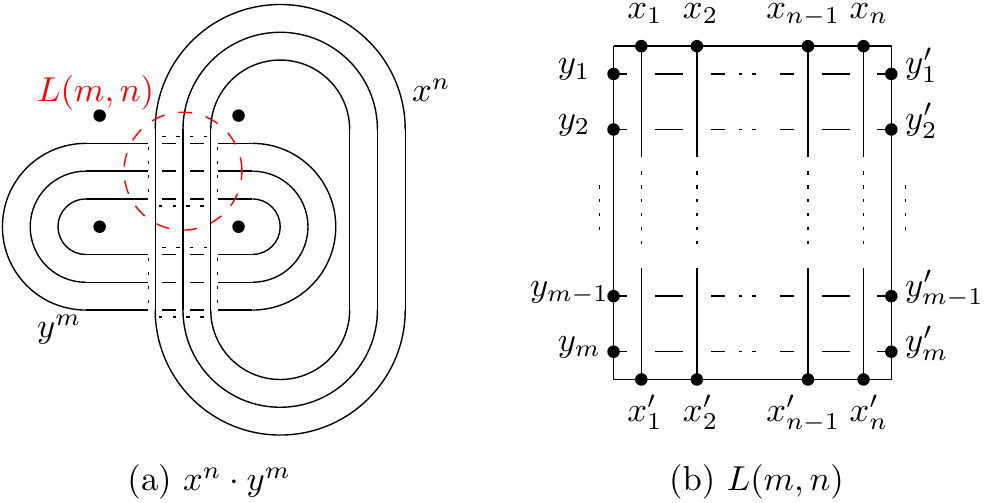}
\caption{Product of curves in $F_{0,4}$ and lattice crossing $L(m,n)$}
\label{fig:LC_intro}
\end{figure}

Since the algebraic structure of KBSA of $F_{g,n}$ is a subject of recent research interests (see \cites{Prz1999-2,PS2000,PS2019,FK2018,Le2018,FKL2019,BW2016}), determining formulas for products of curves on $F_{g,n}$ is an important problem. To understand such products locally, in \cite{DLP2015} authors considered lattice crossing $L(m,n)$ (see Figure~\ref{fig:LC_intro}) and asked if closed-form formulas for coefficients $C(A)$ of Catalan states $C$ of $L(m,n)$ in the Relative Kauffman Bracket Skein Module (RKBSM) of a cylinder with $2(m+n)$ points fixed on its boundary can be found. The problem of finding these coefficients led to some new and interesting developments. For instance, plucking polynomial of plane rooted trees was introduced by J.H.~Przytycki in \cites{Prz2016-2,Prz2016}\footnote{Plucking polynomial was introduced by J.H.~Przytycki during his talk \emph{``$q$-polynomial invariant of rooted trees''}, Knots in Washington XXXVIII: \emph{30 years of the Jones polynomial}, GWU, May 9-11, 2014.} and studied in \cites{CMPWY2017,CMPWY2018,CMPWY2019}. Using a version of this polynomial (for weighted plane rooted trees), a formula for coefficients of Catalan states of $L(m,n)$ with no returns on one side was given in \cite{DP2019}. Finally, formulas for $C(A)$ for some families of Catalan states with returns on all sides were found in \cite{DM2021} and the problem of determining coefficients of Catalan states of B-type lattice crossing was studied in \cite{DR2021}. In this paper, we show that $C(A)$ for any Catalan state $C$ of $L(m,n)$ can be expressed as a linear combination over $\mathbb{Q}(A)$ of coefficients of Catalan states with no top returns. This fact combined with results of \cite{DP2019} gives an efficient method for finding $C(A)$.

The paper is organized as follows. In Section~\ref{s:pre}, we introduce some new terminology describing parts of a Catalan state, such as a \emph{roof state}, a \emph{middle state}, a \emph{floor state}, etc. (see Definition~\ref{def:states_intro}), and give a short summary of results from \cite{DLP2015} and \cite{DP2019}. In Section~\ref{s:1st_row_exp}, we derive formula \eqref{eqn:1st_row_exp} (called the \emph{first-row expansion}) in Proposition~\ref{prop:1st_row_exp} which plays an important role in the next section. We also show that Laurent polynomials $C(A)$ have non-negative integer coefficients (Corollary~\ref{cor:non_negative_coef}) and $C(A) \neq 0$ if and only if the Catalan state $C$ is realizable (Theorem~\ref{thm:realizable_coef_non_zero}). In Section~\ref{s:Theta_state_expansion}, we define \emph{$\Theta_{A}$-state expansion} (Definition~\ref{def:Theta_state_expansion}) and show that it exists for any roof state $R$ (Theorem~\ref{thm:main}). The proof of Theorem~\ref{thm:main} yields an algorithm (Algorithm~\ref{alg:Theta_state_expansion}) for finding $\Theta_{A}$-state expansion for $R$ which allows us to compute $C(A)$ for any Catalan state $C$ of $L(m,n)$ with $R$ as its roof state. We conclude this paper by presenting applications of Algorithm~\ref{alg:Theta_state_expansion} (see Example~\ref{ex:Theta_state_expansion_2} and Corollary~\ref{cor:formula_T_nju}) and providing an example of a Catalan state $C$ in $L(9,5)$ with non-unimodal coefficient $C(A)$ (see Example~\ref{ex:Theta_state_expansion_3}).

\section{Preliminaries}
\label{s:pre}

Let $\mathrm{R}^{2}_{m,n,2k-n}$ be a rectangle with $2(m+k)$ points fixed on its boundary as shown in Figure~\ref{fig:states_intro}(a). A \emph{crossingless connection} $C$ in $\mathrm{R}^{2}_{m,n,2k-n}$ (or simply \emph{crossingless connection}) consists of $(m+k)$ arcs embedded in the rectangle that join points that are fixed on its boundary. For a crossingless connection $C$ in $\mathrm{R}^{2}_{m,n,2k-n}$, we will refer to points fixed on the top side (respectively, bottom, left, and right side) of $\mathrm{R}^{2}_{m,n,2k-n}$ as its top-boundary points (respectively, bottom-, left-, and right-boundary points) and we denoted by $\mathrm{ht}(C), n_{t}(C)$, $n_{b}(C)$ the number of left-, top-, bottom-boundary points of $C$. We label these points using the labels of corresponding points ($x_{i}, x'_{i}, y_{j}, y'_{j}$) fixed on the boundary of $\mathrm{R}^{2}_{m,n,2k-n}$.

An arc of $C$ that joins a pair of top-boundary points is called a \emph{top return}. Analogously, we can define \emph{bottom, left}, and \emph{right returns}.

\begin{definition}
\label{def:states_intro}
A \emph{roof state} (respectively \emph{floor state}) is a crossingless connection with no bottom (respectively top) returns (see Figure~\ref{fig:states_intro}(b) and (c)). A roof state with no top returns is called a \emph{middle state} (see Figure~\ref{fig:states_intro}(d)). A roof state $R$ with $\mathrm{ht}(R) = 0$ is called a \emph{top state}, likewise, a floor state $F$ with $\mathrm{ht}(F) = 0$ is called a \emph{bottom state}.
\end{definition}

\begin{figure}[ht] 
\centering
\includegraphics[scale=1]{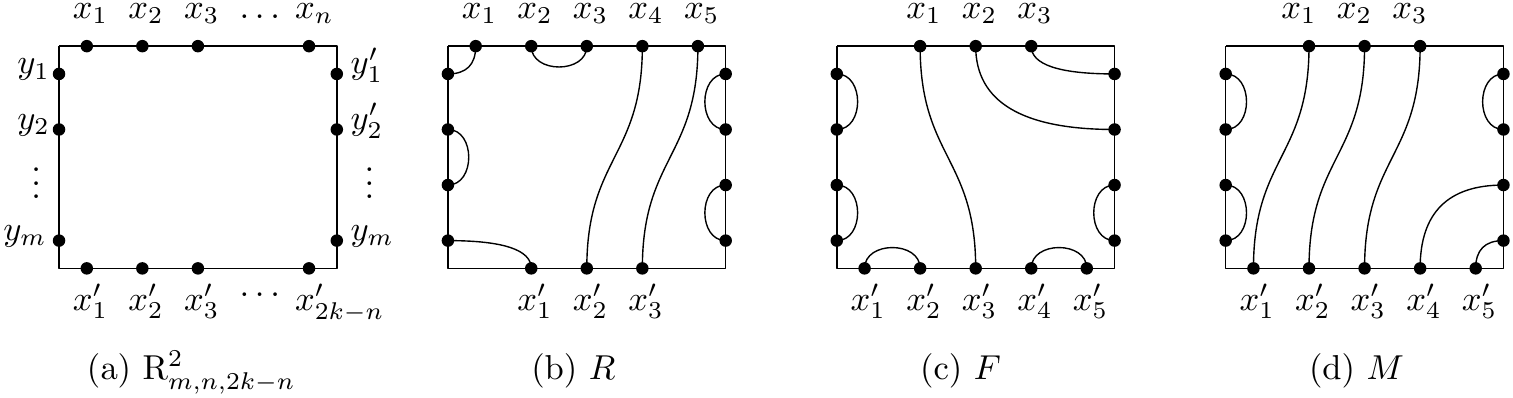}
\caption{Rectangle $\mathrm{R}^{2}_{m,n,2k-n}$, roof state $R$, floor state $F$, and middle state $M$}
\label{fig:states_intro}
\end{figure}

For a pair of tangles $T_{1}$ and $T_{2}$, their \emph{vertical product} $T_{1} *_{v} T_{2}$ is the tangle shown in Figure~\ref{fig:vprod_hprod} whenever the numbers of bottom-boundary points of $T_{1}$ and top-boundary points of $T_{2}$ are equal. Otherwise $T_{1} *_{v} T_{2}$ is not a tangle and in such a case we set $T_{1} *_{v} T_{2} = K_{0}$. Thus, naturally we define
\begin{equation*}
K_{0} *_{v} T_{2} = T_{1} *_{v} K_{0} = K_{0} *_{v} K_{0} = K_{0}.
\end{equation*}
Horizontal product $T_{1} *_{h} T_{2}$ of tangles $T_{1}$ and $T_{2}$ can be defined in an analogous way to the vertical product (see Figure~\ref{fig:vprod_hprod}). Moreover, as it could easily be seen,
\begin{equation*}
(T_{11} *_{h} T_{12}) *_{v} (T_{21} *_{h} T_{22}) = (T_{11} *_{v} T_{21}) *_{h} (T_{12} *_{v} T_{22})
\end{equation*}
for any tangles $T_{11},T_{12},T_{21},T_{22}$ provided that all vertical and horizontal products above are not $K_{0}$. Furthermore, for crossingless connections $C_{1}, C_{2}$ and a floor state $F$, if $C_{1} *_{v} F = C_{2} *_{v} F \neq K_{0}$, then $C_{1} = C_{2}$. Analogously, if $R *_{v} C_{1} = R *_{v} C_{2} \neq K_{0}$ for a roof state $R$, then also $C_{1} = C_{2}$. 

\begin{figure}[ht] 
\centering
\includegraphics[scale=1]{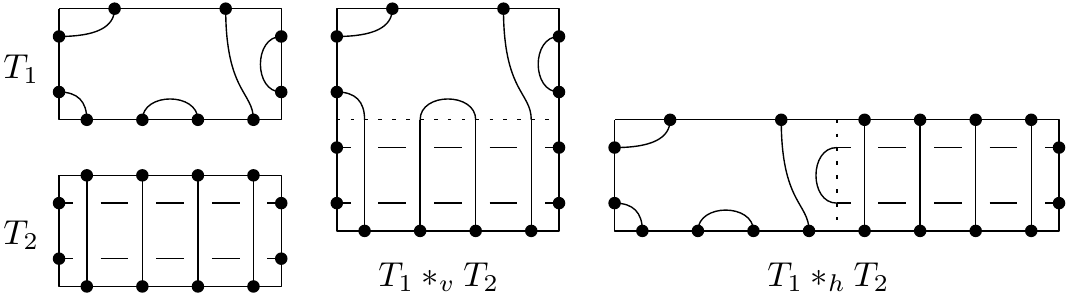}
\caption{Vertical and horizontal products of $T_{1}$ and $T_{2}$}
\label{fig:vprod_hprod}
\end{figure}

A crossingless connection $C$ is called a \emph{Catalan state} if $n_{t}(C) = n_{b}(C)$, and we denote by $\mathrm{Cat}(m,n)$ the set of all such states in $\mathrm{R}^{2}_{m,n,n}$. 

Recall after \cite{DLP2015}, an $m \times n$-\emph{lattice crossing} $L(m,n)$ is a $2(m+n)$-tangle obtained by placing $n$ parallel vertical line segments above $m$ parallel horizontal line segments (see Figure~\ref{fig:LC_intro}(b)). A Kauffman state $s$ of $L(m,n)$ is an assignment of positive and negative markers to all of its crossings.\footnote{A Kauffman state $s$ of $L(m,n)$ can also be viewed as an $m \times n$ matrix with entries $\pm{1}$.} Let $\mathcal{K}(m,n)$ be the set of all Kauffman states of $L(m,n)$. For $s \in \mathcal{K}(m,n)$, we denote by $D_{s}$ the diagram obtained from $L(m,n)$ after smoothing all of its crossings according to the rule shown in Figure~\ref{fig:markers}. Let $C_{s}$ be the Catalan state obtained from $D_{s}$ after removing all of its closed components. We say that a Catalan state $C$ is \emph{realizable} if $C = C_{s}$ for some $s \in \mathcal{K}(m,n)$.

\begin{figure}[ht] 
\centering
\includegraphics[scale=1]{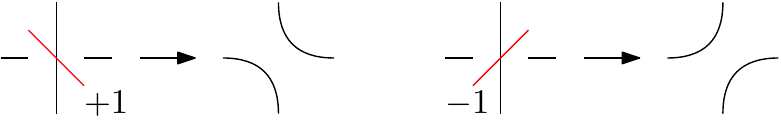}
\caption{Positive and negative markers}
\label{fig:markers}
\end{figure}

As it was shown in \cite{Prz1999}, $\mathrm{Cat}(m,n)$ is a basis for RKBSM of a cylinder with $2(m+n)$ points fixed on its boundary. Therefore,
\begin{equation}
\label{eqn:L_mn}
L(m,n) = \sum_{C \in \mathrm{Cat}(m,n)} C(A) \, C.
\end{equation}
In the linear combination above, the Laurent polynomials $C(A)$ are called \emph{coefficients of Catalan states $C$ of lattice crossing $L(m,n)$}. Let $|D_{s}|$ be the number of closed components of $D_{s}$, and let $p(s)$ and $n(s)$ be the number of positive and negative markers of $s$, respectively. Denote by $\mathcal{K}(C)$ the set of all $s \in \mathcal{K}(m,n)$ such that $C_{s} = C$. Then $C(A)$ is given by
\begin{equation}
C(A) = \sum_{s \in \mathcal{K}(C)} A^{p(s)-n(s)} (-A^{2}-A^{-2})^{|D_{s}|}.
\label{eqn:C(A)}
\end{equation}

We extend the definition of the coefficient of a Catalan state of a lattice crossing as follows. Let $C$ be a crossingless connection or $C = K_{0}$, define $[[C]]_{A}$ by
\begin{equation*}
[[C]]_{A} = 
\begin{cases}
C(A), & \text{if}\ C \ \text{is a Catalan state}, \\
0,    & \text{otherwise}.
\end{cases}
\end{equation*}
For a crossingless connection $C$, denote by $\overline{C}$ its reflection about a vertical line and by $C^{*}$ its $\pi$-rotation. We also put $\overline{K_{0}} = K_{0}$ and $K_{0}^{*} = K_{0}$.

\begin{remark} 
\label{rem:prop_coef}
As one can check, for any crossingless connection $C$ or $C = K_{0}$,
\begin{enumerate}
\item[i)] $[[\overline{C}]]_{A} = [[C]]_{A^{-1}}$ and
\item[ii)] $[[C^{*}]]_{A} = [[C]]_{A}$.
\end{enumerate}
\end{remark}

Let $l^{h}_{i}$ and $l^{v}_{j}$ be horizontal and vertical lines in $\mathrm{R}^{2}_{m,n,n}$ that are shown in Figure~\ref{fig:lhlv_splitting}(a). Note that, horizontal lines $l^{h}_{i}$ are also well-defined in $\mathrm{R}^{2}_{m,n,2k-n}$ when $k \neq n$. Given a crossingless connection $C$ and a horizontal line $l = l^{h}_{i}$, using a regular isotopy (that keeps boundary points fixed) we deform $C$ so that $l$ intersects $C$ at the minimal number of points and let us denote this number by $\#(C \cap l)$. 

\begin{figure}[ht] 
\centering
\includegraphics[scale=1]{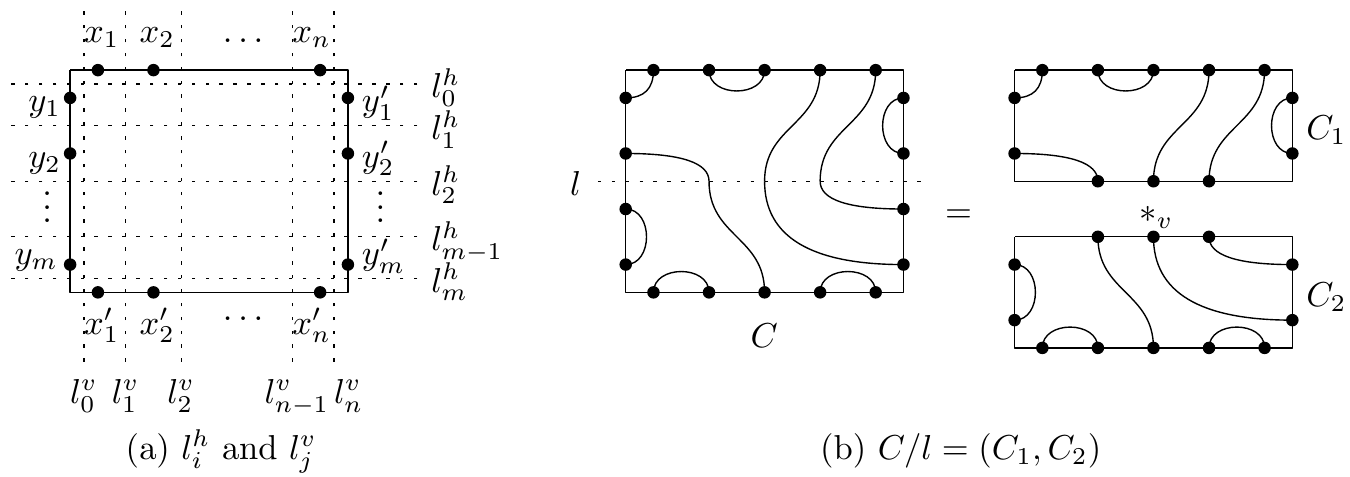}
\caption{Horizontal lines $l^{h}_{i}$ and vertical lines $l^{v}_{j}$; splitting $C/l$ of $C$ along $l$}
\label{fig:lhlv_splitting}
\end{figure}

\begin{theorem}[\protect\cite{DLP2015}, Theorem~2.5] 
\label{thm:vh_line_condi}
A Catalan state $C \in \mathrm{Cat}(m,n)$ is realizable if and only if $\#(C \cap l^{h}_{i}) \leq n$ for $i = 1,2,\ldots,m-1$ and $\#(C \cap l^{v}_{j}) \leq m$ for $j = 1,2,\ldots,n-1$.
\end{theorem}

Define \emph{splitting $C/l$} of a crossingless connection $C$ along a horizontal line $l$ to be a pair $(C_{1},C_{2})$ of crossingless connections, such that $C = C_{1} *_{v} C_{2}$ and $n_{b}(C_{1}) = \#(C \cap l)$ (see Figure~\ref{fig:lhlv_splitting}(b)).

\begin{theorem}[\protect\cite{DP2019}, Theorem~5.1] 
\label{thm:split_cat}
Given a crossingless connection $C$ and a horizontal line $l$ with $\#(C \cap l) \geq \min\{n_{t}(C),n_{b}(C)\}$, let $(C_{1},C_{2})$ be the splitting $C/l$ of $C$ along $l$. Then
\begin{equation*}
[[C]]_{A} = [[C_{1}]]_{A} \, [[C_{2}]]_{A}.
\end{equation*}
\end{theorem}

\begin{remark}
\label{rem:hprod}
An analogous result to Theorem~\ref{thm:split_cat} holds for a Catalan state $C$ and a vertical line $l = l^{v}_{j}$ with $\#(C \cap l) = \mathrm{ht}(C)$.
\end{remark}

Let $C \in \mathrm{Cat}(m,n)$ be a Catalan state with no bottom returns. Following the construction given in \cite{DP2019}, one defines a dual plane rooted tree $T(C)$ with root $v_{0}$ and a weight function $\alpha$ (called \emph{delay})\footnote{The delay is defined from the set of leaves of $T(C)$ (not including the root) to $\mathbb{N}$ by $\alpha(v) = k$ if $v$ corresponds to a left or right return with its lower end $y_{k}$ or $y'_{k}$ and $\alpha(v) = 1$ otherwise.} as shown in Figure~\ref{fig:rooted_tree}.

\begin{figure}[ht] 
\centering
\includegraphics[scale=1]{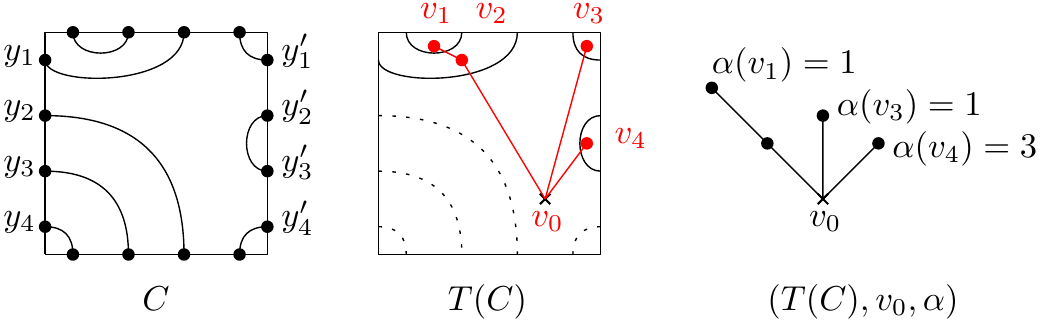}
\caption{Plane rooted tree $(T(C),v_{0})$ with delay $\alpha$}
\label{fig:rooted_tree}
\end{figure}

\begin{definition}[\cite{DP2019}, Definition~3.1] 
\label{def:q_poly}
Let $(T,v_{0})$ be a plane rooted tree $T$ with root $v_{0}$ and let $\alpha$ be a weight function from the set of leaves of $T$ (not including $v_{0}$) to $\mathbb{N}$. Denote by $L_{1}(T)$ the set of all leaves $v$ of $T$ with $\alpha(v) = 1$. The \emph{plucking polynomial} $Q_{q}$ of $(T,v_{0},\alpha)$ is a polynomial in variable $q$ defined as follows: For $T$ with no edges $Q_{q}(T,v_{0},\alpha) = 1$, and otherwise
\begin{equation*}
Q_{q}(T,v_{0},\alpha) = \sum_{v \in L_{1}(T)} q^{r(T,v_{0},v)} Q_{q}(T-v,v_{0},\alpha_{v}),
\end{equation*}
where $r(T,v_{0},v)$ is the number of vertices of $T$ to the right of the unique path connecting $v_{0}$ and $v$, and $\alpha_{v}$ is the weight function defined by $\alpha_{v}(u) = \max\{1,\alpha(u)-1\}$ if $u$ is a leaf of $T$ and $\alpha_{v}(u) = 1$ if $u$ is a new leaf of $T-v$.
\end{definition}

As shown in \cite{DP2019}, coefficients of realizable Catalan states of $L(m,n)$ with no bottom returns can be found using Kauffman states with rows $r_{i}$, $1 \leq i \leq m$, in the form
\begin{equation*}
r_{i} = ( \underset{b_{i}}{\underbrace{1,1,\ldots,1}},\underset{n-b_{i}}{\underbrace{-1,-1,\ldots,-1}}),
\end{equation*}
where $0 \leq b_{i} \leq n$. Hence, these states are uniquely determined by sequences $\mathbf{b} = (b_{1},b_{2},\ldots,b_{m})$. Denote by $\mathfrak{b}(C)$ the set of all such sequences that realize $C$ and let $\beta(C) = \max\{\Vert\mathbf{b}\Vert: \mathbf{b} \in \mathfrak{b}(C)\}$, where $\Vert\mathbf{b}\Vert = b_{1} + b_{2} + \cdots + b_{m}$.

\begin{theorem}[\cite{DP2019}, Theorem~3.4] 
\label{thm:coef_no_bot_rtn}
Let $C \in \mathrm{Cat}(m,n)$ be a realizable Catalan state with no bottom returns, and let $T = T(C)$ be its corresponding plane rooted tree with root $v_{0}$ and delay $\alpha$. Then 
\begin{equation*}
C(A) = A^{2\beta(C)-mn} Q_{A^{-4}}^{*}(T,v_{0},\alpha),
\end{equation*}
where $Q_{q}^{*}(T,v_{0},\alpha) = q^{-\min\deg_{q}(Q_{q}(T,v_{0},\alpha))} \, Q_{q}(T,v_{0},\alpha)$.
\end{theorem}

\section{First-row Expansion}
\label{s:1st_row_exp}

Let $\mathcal{B}_{n}$ be the set of all bottom states $F$ with $n_{b}(F) = n$.\footnote{One can show that the number of such states is $|\mathcal{B}_{n}| = \binom{n}{\left\lfloor \frac{n}{2}\right\rfloor}$.} Define a map $\varphi_{n}$ from $\mathcal{B}_{n}$ to the set $\mathrm{Fin}(\mathbb{N})$ of all finite subsets of $\mathbb{N}$ by $\varphi_{n}(F) = \{i_{1},i_{2},\ldots,i_{k}\}$, where $i_{1} < i_{2} < \cdots < i_{k}$ are the indices of the left ends of bottom returns of $F$. For example, $\varphi_{7}(F) = \{3,4\}$ for the bottom state $F$ shown in Figure~\ref{fig:bot_state}.

\begin{figure}[ht] 
\centering
\includegraphics[scale=1]{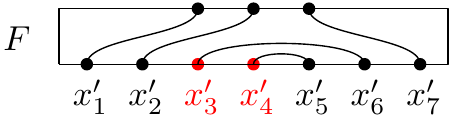}
\caption{Indices of bottom returns of $F$}
\label{fig:bot_state}
\end{figure}

Let $\mathcal{L}_{n} = \varphi_{n}(\mathcal{B}_{n})$, then clearly $\varphi_{n}: \mathcal{B}_{n} \to \mathcal{L}_{n}$ is a bijection. Thus, for each $I \in \mathcal{L}_{n}$ there is a corresponding bottom state $\varphi_{n}^{-1}(I)$. Clearly, $\varphi_{n}^{-1}(\emptyset) = L(0,n)$ and  $\varphi_{n}^{-1}(\{i\})$ for $1 \leq i \leq n-1$ is given in Figure~\ref{fig:phi_inverse}. We show how to find $\varphi_{n}^{-1}(I)$ for an arbitrary $I \in \mathcal{L}_{n}$ in Remark~\ref{rem:oplus}.

\begin{figure}[ht] 
\centering
\includegraphics[scale=1]{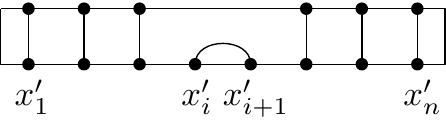}
\caption{$\varphi_{n}^{-1}(\{i\})$ for $1 \leq i \leq n-1$}
\label{fig:phi_inverse}
\end{figure}

Let $C$ be a crossingless connection with $n_{b}(C) = n$ or $C = K_{0}$, and $I \in \mathrm{Fin}(\mathbb{N})$. If $C \neq K_{0}$ and $I \in \mathcal{L}_{n}$, define $C_{I}$ as the unique crossingless connection $C'$ such that 
\begin{equation*}
C = C' *_{v} \varphi_{n}^{-1}(I)
\end{equation*}
provided that $C'$ exists, and set $C_{I} = K_{0}$ otherwise.
If $C = K_{0}$ or $I \notin \mathcal{L}_{n}$, we also set $C_{I} = K_{0}$. Geometrically, $C_{I}$ is simply a crossingless connection obtained from $C$ after removing its all bottom returns indexed by $I$ together with their ends (see Figure~\ref{fig:C_I}). 

\begin{figure}[ht] 
\centering
\includegraphics[scale=1]{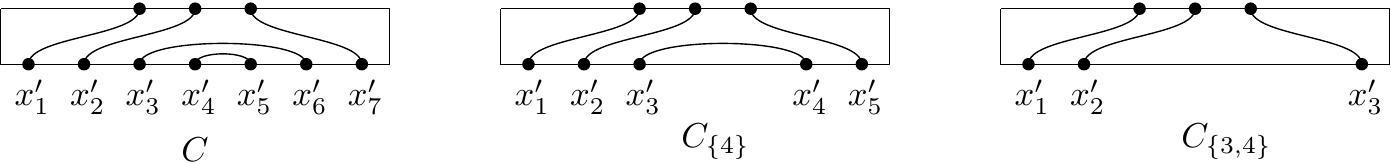}
\caption{Crossingless connections $C$ and $C_{I}$ for $I = \{4\},\{3,4\}$}
\label{fig:C_I}
\end{figure}

Define a partial order on $\mathcal{L}_{n}$ as follows: for $I,J \in \mathcal{L}_{n}$, $J \preceq_{n} I$ if $(\varphi_{n}^{-1}(I))_{J} \neq K_{0}$ (see Figure~\ref{fig:poset_Ln}). Notice that, if $J \preceq_{n} I$ then
\begin{eqnarray*}
\varphi_{n+1}^{-1}(I) 
&=& \varphi_{n}^{-1}(I) *_{h} L(0,1) 
= ((\varphi_{n}^{-1}(I))_{J} *_{v} \varphi_{n}^{-1}(J)) *_{h} (L(0,1) *_{v} L(0,1)) \\
&=& ((\varphi_{n}^{-1}(I))_{J} *_{h} L(0,1)) *_{v} (\varphi_{n}^{-1}(J) *_{h} L(0,1)) \\
&=& ((\varphi_{n}^{-1}(I))_{J} *_{h} L(0,1)) *_{v} \varphi_{n+1}^{-1}(J).
\end{eqnarray*}
It follows that, if $J \preceq_{n} I$ in $\mathcal{L}_{n}$ then $J \preceq_{n+1} I$ in $\mathcal{L}_{n+1}$. Therefore, $(\mathcal{L}_{n},\preceq_{n})$ is a subposet of $(\mathcal{L}_{n+1},\preceq_{n+1})$ and thus one defines a partial order $\preceq$ on $\mathrm{Fin}(\mathbb{N})$ by: $I \preceq J$ if $I \preceq_{n} J$ for some $n$.

\begin{figure}[ht] 
\centering
\includegraphics[scale=1]{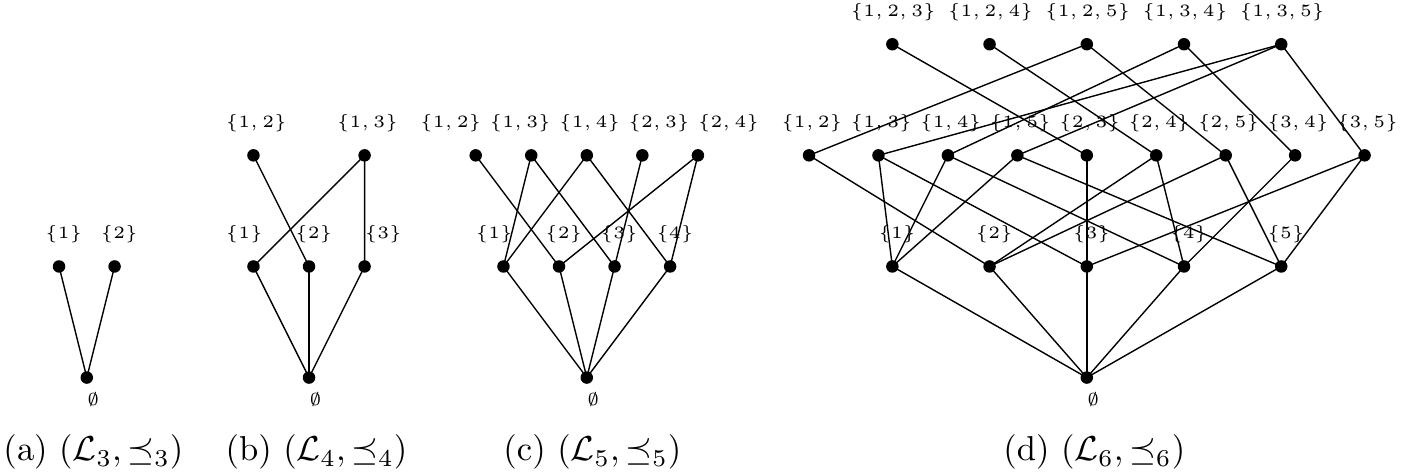}
\caption{Posets $(\mathcal{L}_{n},\preceq_{n})$ for $n=3,4,5,6$}
\label{fig:poset_Ln}
\end{figure}

Clearly, $I = \{i\} \in \mathcal{L}_{i+1}$, and if $I \in \mathcal{L}_{m}$ for some $m$ and $I \neq \emptyset$ then $I \cup \{j\} \in \mathcal{L}_{m+1}$ for any $j < \min I$. Therefore, given a non-empty finite subset $I$ of $\mathbb{N}$, $I \in \mathcal{L}_{\max I+|I|}$. Hence, there is a non-negative integer $n_{I}$ such that, $I \in \mathcal{L}_{n}$ for $n \geq n_{I}$ and $I \notin \mathcal{L}_{n}$ for $n < n_{I}$.

For $I,J \in  \mathrm{Fin}(\mathbb{N})$, let $n^{*} = \max\{n_{I}+2|J|,n_{J}\}$. Define an operation $I \oplus J$ by
\begin{equation*}
I \oplus J = \varphi_{n^{*}}(  \varphi_{n^{*}-2|J|}^{-1}(I) *_{v} \varphi_{n^{*}}^{-1}(J) ).
\end{equation*}
The geometric interpretation of $I \oplus J$ is given in Figure~\ref{fig:I_oplus_J}.

\begin{figure}[ht] 
\centering
\includegraphics[scale=1]{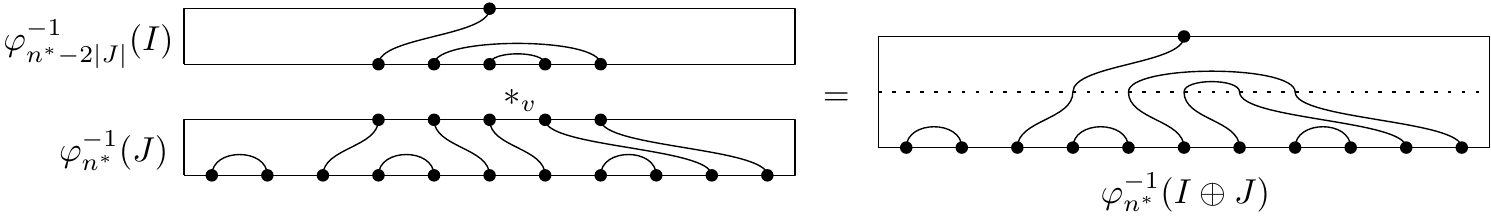}
\caption{Geometric interpretation of $I \oplus J$}
\label{fig:I_oplus_J}
\end{figure}

\begin{remark}
\label{rem:i_oplus_j}
Based on the geometric interpretation of $I \oplus J$ given above, it can easily be seen that
\begin{equation*}
\{i\} \oplus \{j\}
= \begin{cases}
\{i,j\}, & \text{if}\ i < j, \\
\{j,i+2\}, & \text{otherwise}, 
\end{cases}
=
\begin{cases}
\{j-2\} \oplus \{i\}, & \text{if} \ i \leq j-2, \\
\{j\} \oplus \{i+2\}, & \text{if} \ i \geq j.
\end{cases}
\end{equation*}
\end{remark}

\begin{proposition}
\label{prop:prop_oplus}
Let $I,J \in \mathrm{Fin}(\mathbb{N})$.
\begin{enumerate}
\item[i)] $I \oplus J = I \cup J$ if either $I = \emptyset$ or $J = \emptyset$, or both $I$ and $J$ are nonempty and $\max I < \min J$.
\item[ii)] $I \oplus J \in \mathcal{L}_{n}$ if and only if $n \geq \max\{n_{I}+2|J|,n_{J}\}$, i.e., $n_{I \oplus J} = \max\{n_{I}+2|J|,n_{J}\}$.
\item[iii)] If $n \geq n_{I \oplus J}$, then $\varphi_{n}^{-1}(I \oplus J) =  \varphi_{n-2|J|}^{-1}(I) *_{v} \varphi_{n}^{-1}(J)$. 
\item[iv)] If $I \in \mathcal{L}_{n-2|J|}$ and $J \in \mathcal{L}_{n}$, then $I \oplus J \in \mathcal{L}_{n}$.
\end{enumerate}
\end{proposition}

\begin{proof}
Let $n^{*} = \max\{n_{I}+2|J|,n_{J}\}$. For i), if either $I = \emptyset$ or $J = \emptyset$, then clearly $I \oplus J = I \cup J$. Thus, assume that $I$ and $J$ are nonempty and notice that $J$ is a subset of indices of bottom returns of $F = \varphi_{n^{*}-2|J|}^{-1}(I) *_{v} \varphi_{n^{*}}^{-1}(J)$. Since for each $i = 1,2,\ldots,\min J -1$ there is an arc joining $x_{i}$ and $x'_{i}$ in $\varphi_{n^{*}}^{-1}(J)$ and $\max I \leq \min J -1$, it follows that indices of bottom returns of $\varphi_{n^{*}-2|J|}^{-1}(I)$ are among the indices of bottom returns of $F$. Therefore, $I \cup J \subseteq I \oplus J$. Since $|I \oplus J| = |I|+|J| = |I \cup J|$, we see that $I \oplus J = I \cup J$.

\medskip

For ii), by definition $I \oplus J \in \mathcal{L}_{n^{*}}$, so $I \oplus J \in \mathcal{L}_{n}$ for any $n \geq n^{*}$. To prove ii), it suffices to show that if $n < n^{*}$ then $I \oplus J \notin \mathcal{L}_{n}$. Indeed, either $n^{*}-2|J| = n_{I}$ or $n^{*} = n_{J}$. Therefore, $x'_{n^{*}-2|J|}$ is a right end of a bottom return in $\varphi_{n^{*}-2|J|}^{-1}(I)$ or $x'_{n^{*}}$ is a right end of a bottom return in $\varphi_{n^{*}}^{-1}(J)$. Hence, $x'_{n^{*}}$ is a right end of a bottom return in $\varphi_{n^{*}}^{-1}(I \oplus J)$. It follows that $I \oplus J \notin \mathcal{L}_{n}$ for any $n < n^{*}$.

\medskip

For iii), we see that by ii) $n_{I \oplus J} = n^{*}$, so $I \in \mathcal{L}_{n^{*}-2|J|} \subseteq \mathcal{L}_{n-2|J|}$ and $J \in \mathcal{L}_{n^{*}} \subseteq \mathcal{L}_{n}$ for $n \geq n_{I \oplus J}$, and therefore $\varphi_{n-2|J|}^{-1}(I)$ and $\varphi_{n}^{-1}(J)$ are bottom states. Consequently,
\begin{eqnarray*}
\varphi_{n}^{-1}(I \oplus J)
&=& \varphi_{n^{*}}^{-1}(I \oplus J) *_{h} L(0,n-n^{*})
= (\varphi_{n^{*}-2|J|}^{-1}(I) *_{v} \varphi_{n^{*}}^{-1}(J)) *_{h} (L(0,n-n^{*}) *_{v} L(0,n-n^{*}) ) \\
&=& (\varphi_{n^{*}-2|J|}^{-1}(I) *_{h} L(0,n-n^{*})) *_{v} (\varphi_{n^{*}}^{-1}(J) *_{h} L(0,n-n^{*})) \\
&=& \varphi_{n-2|J|}^{-1}(I) *_{v} \varphi_{n}^{-1}(J).
\end{eqnarray*}

\medskip

For iv), we notice that, since $I \in \mathcal{L}_{n-2|J|}$ and $J \in \mathcal{L}_{n}$, by definition of $n_{I}$ and $n_{J}$, $n-2|J| \geq n_{I}$ and $n \geq n_{J}$, i.e., $n \geq \max\{n_{I}+2|J|,n_{J}\}$. Therefore, using ii) we conclude that $I \oplus J \in \mathcal{L}_{n}$.
\end{proof}

\begin{remark}
\label{rem:oplus}
Let $I = \{i_{1}, i_{2}, \ldots, i_{t}\} \in \mathcal{L}_{n}$ with $I \neq \emptyset$, then by Proposition~\ref{prop:prop_oplus}(i) and (iii), we see that
\begin{equation*}
\varphi_{n}^{-1}(I) = \varphi_{n}^{-1}(\{i_{1}\} \oplus \{i_{2}\} \oplus \cdots \oplus \{i_{t}\}) = \varphi_{n-2t+2}^{-1}(\{i_{1}\}) *_{v} \varphi_{n-2t+4}^{-1}(\{i_{2}\}) *_{v} \cdots *_{v} \varphi_{n}^{-1}(\{i_{t}\}).
\end{equation*}
Thus, $\varphi_{n}^{-1}(I)$ can be found for all $I \in \mathcal{L}_{n}$. Furthermore, since $n_{I} = \max I'$, where $I'$ is the set of indices of right ends of the bottom returns of $\varphi_{i_{t}+|I|}^{-1}(I)$, we can also find $n_{I}$.
\end{remark}

\begin{lemma} 
\label{lem:F_J_I}
For every roof state $R$, floor state $F$, crossingless connection $C$, and $I,J \in \mathrm{Fin}(\mathbb{N})$,
\begin{equation}
\label{eqn:M_vprod_F_I}
(R *_{v} F)_{I} = R *_{v} F_{I}
\end{equation}
and 
\begin{equation}
\label{eqn:F_J_I}
(C_{J})_{I} = C_{I \oplus J}.
\end{equation}
\end{lemma}

\begin{proof}
Let $n = n_{b}(F)$. If $R *_{v} F = K_{0}$ then $n_{b}(R) \neq n_{t}(F)$. Since either $F_{I} = K_{0}$ or $F_{I}$ is a floor state with $n_{t}(F_{I}) = n_{t}(F)$, it follows that $R *_{v} F_{I} = K_{0}$. If $R *_{v} F$ is a crossingless connection, $R *_{v} F$ and $F$ have same bottom returns and then one argues that $(R *_{v} F)_{I} = K_{0}$ if and only if $F_{I} = K_{0}$. Therefore, it suffices to show \eqref{eqn:M_vprod_F_I} when $R *_{v} F \neq K_{0}$ and $F_{I} \neq K_{0}$. Indeed, in such a case,
\begin{equation*}
R *_{v} F = R *_{v} F_{I} *_{v} \varphi_{n}^{-1}(I),
\end{equation*}
so $(R *_{v} F)_{I} = R *_{v} F_{I}$.

To prove \eqref{eqn:F_J_I}, we see that by using \eqref{eqn:M_vprod_F_I} it suffices to assume that $C = F$ is a bottom state. Notice that, $(\varphi_{n'}^{-1}(I \oplus J))_{J} = \varphi_{n'-2|J|}^{-1}(I) \neq K_{0}$ for $n' = n_{I \oplus J}$, so $J \preceq_{n'} I \oplus J$ and consequently $J \preceq I \oplus J$. Hence, if $J \npreceq \varphi_{n}(F)$ then $I \oplus J \npreceq \varphi_{n}(F)$, and therefore
\begin{equation*}
(F_{J})_{I} = (K_{0})_{I} = K_{0} = F_{I \oplus J}.
\end{equation*}
Assume that $J \preceq \varphi_{n}(F)$. If $I \npreceq \varphi_{n-2|J|}(F_{J})$, we show that $I \oplus J \npreceq \varphi_{n}(F)$, so consequently 
\begin{equation*}
(F_{J})_{I} = K_{0} = F_{I \oplus J}. 
\end{equation*}
Suppose that $I \oplus J \preceq \varphi_{n}(F)$, then by Proposition~\ref{prop:prop_oplus}(ii), $n \geq n_{I \oplus J}$. Since by the definition and Proposition~\ref{prop:prop_oplus}(iii)
\begin{equation*}
F_{J} *_{v} \varphi_{n}^{-1}(J) = F = F_{I \oplus J} *_{v} \varphi_{n}^{-1}(I \oplus J) = F_{I \oplus J} *_{v} \varphi_{n-2|J|}^{-1}(I) *_{v} \varphi_{n}^{-1}(J),
\end{equation*}
it follows that
\begin{equation*}
F_{J} = F_{I \oplus J} *_{v} \varphi_{n-2|J|}^{-1}(I).
\end{equation*}
Therefore, $(F_{J})_{I} = F_{I \oplus J} \neq K_{0}$, i.e., $I \preceq \varphi_{n-2|J|}(F_{J})$, a contradiction. Finally, if $I \preceq \varphi_{n-2|J|}(F_{J})$, then $n \geq n_{I \oplus J}$ and by Proposition~\ref{prop:prop_oplus}(iii),
\begin{equation*}
F = F_{J} *_{v} \varphi_{n}^{-1}(J) = (F_{J})_{I} *_{v} \varphi_{n-2|J|}^{-1}(I) *_{v} \varphi_{n}^{-1}(J) = (F_{J})_{I} *_{v} \varphi_{n}^{-1}(I \oplus J),
\end{equation*}
hence, $F_{I \oplus J} = (F_{J})_{I}$.
\end{proof}

For $I,J \in \mathrm{Fin}(\mathbb{N})$ with $J \preceq I$, define
\begin{equation*}
I \ominus J = \varphi_{n_{I}-2|J|}( (\varphi_{n_{I}}^{-1}(I) )_{J} ).
\end{equation*}
The geometric interpretation of $I \ominus J$ is given in  Figure~\ref{fig:I_ominus_J}.

\begin{figure}[ht] 
\centering
\includegraphics[scale=1]{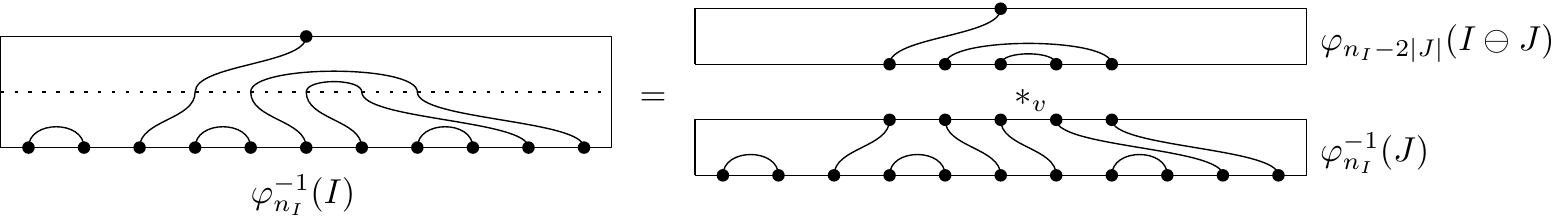}
\caption{Geometric interpretation of $I \ominus J$}
\label{fig:I_ominus_J}
\end{figure}

\begin{proposition}
\label{prop:prop_ominus}
Let $I,J,K \in \mathrm{Fin}(\mathbb{N})$ and $J \preceq I$.
\begin{enumerate}
\item[i)] If $I \in \mathcal{L}_{n}$, then $I \ominus J \in \mathcal{L}_{n-2|J|}$, so consequently  $\varphi_{n-2|J|}^{-1}(I \ominus J) = (\varphi_{n}^{-1}(I) )_{J}$. 
\item[ii)] 
$J \preceq K \oplus I$ and $(K \oplus I) \ominus J = K \oplus (I \ominus J)$. In particular, $(K \oplus I) \ominus I = K$.
\item[iii)]
$(I \ominus J) \oplus J = I$.
\item[iv)]
$J \oplus K \preceq I \oplus K$ and $(I \oplus K) \ominus (J \oplus K) = I \ominus J$.
\end{enumerate}
\end{proposition}

\begin{proof}
For i), since $n \geq n_{I}$, $I \ominus J \in \mathcal{L}_{n_{I}-2|J|} \subseteq \mathcal{L}_{n-2|J|}$. Consequently,
\begin{eqnarray*}
\varphi_{n-2|J|}^{-1}(I \ominus J) &=& \varphi_{n_{I}-2|J|}^{-1}(I \ominus J) *_{h} L(0,n-n_{I})
= (\varphi_{n_{I}}^{-1}(I))_{J} *_{h} L(0,n-n_{I}) \\
&=& (\varphi_{n_{I}}^{-1}(I) *_{h} L(0,n-n_{I}))_{J} \\
&=& (\varphi_{n}^{-1}(I))_{J}.
\end{eqnarray*}

\medskip

For ii), let $n = n_{K \oplus I}$ then clearly $(\varphi_{n}^{-1}(K \oplus I))_{I} = \varphi_{n-2|I|}^{-1}(K) \neq K_{0}$. Therefore, $I \preceq_{n} K \oplus I$ and consequently $J \preceq I \preceq K \oplus I$. Moreover,
\begin{equation*}
(\varphi_{n}^{-1}(K \oplus I))_{J} *_{v} \varphi_{n}^{-1}(J)
= \varphi_{n}^{-1}(K \oplus I)
= \varphi_{n-2|I|}^{-1}(K) *_{v} \varphi_{n}^{-1}(I) = \varphi_{n-2|I|}^{-1}(K) *_{v} (\varphi_{n}^{-1}(I))_{J} *_{v} \varphi_{n}^{-1}(J),
\end{equation*}
so $(\varphi_{n}^{-1}(K \oplus I))_{J} = \varphi_{n-2|I|}^{-1}(K) *_{v} (\varphi_{n}^{-1}(I))_{J}$. Hence,
\begin{eqnarray*}
\varphi_{n-2|J|}^{-1}((K \oplus I) \ominus J) &=& (\varphi_{n}^{-1}(K \oplus I))_{J} 
= \varphi_{n-2|I|}^{-1}(K) *_{v} (\varphi_{n}^{-1}(I))_{J} \\
&=& \varphi_{n-2|I|}^{-1}(K) *_{v} \varphi_{n-2|J|}^{-1}(I \ominus J) \\
&=& \varphi_{n-2|J|}^{-1}(K \oplus (I \ominus J)),
\end{eqnarray*}
and consequently $(K \oplus I) \ominus J = K \oplus (I \ominus J)$.
In particular, since $I \ominus I = \emptyset$ and $K \oplus \emptyset = K$, taking $J = I$ yields $(K \oplus I) \ominus I = K$.
\medskip

For iii), since
$\varphi_{n_{I}-2|J|}^{-1}(I \ominus J) = (\varphi_{n_{I}}^{-1}(I))_{J}$,
\begin{equation*}
\varphi_{n_{I}}^{-1}((I \ominus J) \oplus J) = \varphi_{n_{I}-2|J|}^{-1}(I \ominus J) *_{v} \varphi_{n_{I}}^{-1}(J) = (\varphi_{n_{I}}^{-1}(I))_{J} *_{v} \varphi_{n_{I}}^{-1}(J) = \varphi_{n_{I}}^{-1}(I).
\end{equation*} 
Therefore, $(I \ominus J) \oplus J = I$.

\medskip

For iv), since $J \preceq I$, by iii) and Proposition~\ref{prop:prop_oplus}(iii), 
\begin{equation*}
\varphi_{n}^{-1}(I) = \varphi_{n}^{-1}((I \ominus J) \oplus J) = \varphi_{n-2|J|}^{-1}(I \ominus J) *_{v} \varphi_{n}^{-1}(J)
\end{equation*} 
for all $n \geq n_{I}$. Since $n_{I} \geq n_{J}$, using Proposition~\ref{prop:prop_oplus}(ii) one shows that
\begin{equation*}
n_{I \oplus K} = \max\{n_{I}+2|K|,n_{K}\} \geq \max\{n_{I}+2|K|,n_{J}+2|K|,n_{K}\} = \max\{n_{I}+2|K|,n_{J \oplus K}\}
\end{equation*}
Therefore, for $n = n_{I \oplus K}$,
\begin{eqnarray*}
\varphi_{n}^{-1}(I \oplus K) &=& \varphi_{n-2|K|}^{-1}(I) *_{v} \varphi_{n}^{-1}(K) = \varphi_{n-2|K|-2|J|}^{-1}(I \ominus J) *_{v} \varphi_{n-2|K|}^{-1}(J) *_{v} \varphi_{n}^{-1}(K) \\
&=& \varphi_{n-2|J \oplus K|}^{-1}(I \ominus J) *_{v} \varphi_{n}^{-1}(J \oplus K).
\end{eqnarray*}
It follows that $J \oplus K \preceq I \oplus K$. Moreover, 
\begin{equation*}
\varphi_{n-2|J \oplus K|}^{-1}((I \oplus K) \ominus (J \oplus K)) = (\varphi_{n}^{-1}(I \oplus K))_{J \oplus K} = \varphi_{n-2|J \oplus K|}^{-1}(I \ominus J),
\end{equation*}
so $(I \oplus K) \ominus (J \oplus K) = I \ominus J$.
\end{proof}

Let $\mathcal{W}$ be the set of all pairs $(R,I)$, where $R$ is a roof state and $I \in \mathrm{Fin}(\mathbb{N})$. Define a relation $\trianglelefteq$ on $\mathcal{W}$ as follows: $(R,I) \trianglelefteq (R',I')$ if
\begin{enumerate}
\item [a)] $I' \ominus I \in \mathcal{L}_{n_{t}(R)}$,
\item [b)] $n_{t}(R)+2|I| = n_{t}(R')+2|I'|$, and 
\item [c)] $n_{b}(R) = n_{b}(R')$.
\end{enumerate}
We also use $(R',I') \trianglerighteq (R,I)$ when $(R,I) \trianglelefteq (R',I')$ and, for any $\mathcal{P}' \subset \mathcal{W}$, we write $\mathcal{P}' \trianglerighteq (R,I)$ if $(R',I') \trianglerighteq (R,I)$ for all $(R',I') \in \mathcal{P}'$. The relation $\trianglelefteq$ is transitive as shown in the following lemma.

\begin{lemma}
\label{lem:transitivity_triangle}
If $(R,I) \trianglelefteq (R',I')$ and $(R',I') \trianglelefteq (R'',I'')$, then $(R,I) \trianglelefteq (R'',I'')$.
\end{lemma}

\begin{proof}
For $(R,I)$ and $(R'',I'')$, conditions b) and c) given in the definition of $\trianglelefteq$ are clearly satisfied. Notice that $I \preceq I' \preceq I''$, so $I'' \ominus I$ is defined. Since $I'' \ominus I' \in \mathcal{L}_{n_{t}(R')} = \mathcal{L}_{n+2|I|-2|I'|} = \mathcal{L}_{n-2|I' \ominus I|}$ and $I' \ominus I \in \mathcal{L}_{n}$, where $n = n_{t}(R)$, it follows by Proposition~\ref{prop:prop_ominus}(iii), (ii), and Proposition~\ref{prop:prop_oplus}(iv) that 
\begin{equation*}
I'' \ominus I = ((I'' \ominus I') \oplus I') \ominus I = (I'' \ominus I') \oplus (I' \ominus I) \in \mathcal{L}_{n},
\end{equation*}
i.e., condition a) also holds. Therefore, $(R,I) \trianglelefteq (R'',I'')$.
\end{proof}

Given $(R,I) \in \mathcal{W}$, let $\Theta_{A}(R,I;\cdot)$ be a function of variable $F$ defined by
\begin{equation}
\label{eqn:Theta_A_def}
\Theta_{A}(R,I;F) = [[R *_{v} F_{I}]]_{A},
\end{equation}
where $F$ is a floor state or $F = K_{0}$.

\begin{proposition}
\label{prop:P_prime_concat_reflect_formula}
Let $(R,I) \in \mathcal{W}$ with $n_{t}(R) = n$ and $n_{b}(R) = k$. Assume that there is a finite collection $\mathcal{P} \subset \mathcal{W}$ of pairs $(R',I')$ with $I \preceq I'$, and rational functions $Q_{R',I'}(A) \in \mathbb{Q}(A)$ for each $(R',I')\in \mathcal{P}$, such that
\begin{equation*}
\Theta_{A}(R,I;\cdot) = \sum_{(R',I') \in \mathcal{P}} Q_{R',I'}(A) \, \Theta_{A}(R',I';\cdot).
\end{equation*}
Then
\begin{equation}
\label{eqn:sum_over_P_prime}
\Theta_{A}(R,I;\cdot) = \sum_{(R',I') \in \mathcal{P}'} Q_{R',I'}(A) \, \Theta_{A}(R',I';\cdot),
\end{equation}
where $\mathcal{P}'$ is the set of all pairs $(R',I') \in \mathcal{P}$ such that $(R',I') \trianglerighteq (R,I)$.
Furthermore, the following are true:
\begin{enumerate}
\item[i)] Let $J \in \mathrm{Fin}(\mathbb{N})$, then $(R,I \oplus J) \trianglelefteq (R',I' \oplus J)$ for all $(R',I') \in \mathcal{P}'$ and
\begin{equation*}
\Theta_{A}(R,I \oplus J;\cdot) = \sum_{(R',I') \in \mathcal{P}'} Q_{R',I'}(A) \, \Theta_{A}(R',I' \oplus J;\cdot).
\end{equation*}
\item[ii)] Let $M$ be a middle state with $n_{t}(M) = k$, then $(R *_{v} M,I) \trianglelefteq (R' *_{v} M,I')$ for all $(R',I') \in \mathcal{P}'$ and
\begin{equation*}
\Theta_{A}(R *_{v} M,I;\cdot) = \sum_{(R',I') \in \mathcal{P}'} Q_{R',I'}(A) \, \Theta_{A}(R' *_{v} M,I';\cdot).
\end{equation*} 
\item[iii)] If $I = \emptyset$ in \eqref{eqn:sum_over_P_prime} then $(\overline{R},\emptyset) \trianglelefteq (\overline{R'},\overline{I'})$ for all $(R',I') \in \mathcal{P}'$ and
\begin{equation*}
\Theta_{A}(\overline{R},\emptyset;\cdot) = \sum_{(R',I') \in \mathcal{P}'} Q_{R',I'}(A^{-1}) \, \Theta_{A}(\overline{R'},\overline{I'};\cdot),
\end{equation*}
where $\overline{R}$ and $\overline{R'}$ are reflections of $R$ and $R'$ about a vertical line, and $\overline{I'} = \varphi_{n}(\overline{\varphi_{n}^{-1}(I')})$.
\end{enumerate}
\end{proposition}

\begin{proof}
Given a floor state $F$ either $R *_{v} F_{I}$ is a Catalan state or $R *_{v} F_{I}$ is not a Catalan state. In the former case, for each $(R',I') \in \mathcal{P} \setminus \mathcal{P}'$ either $F_{I'} = K_{0}$ or $F_{I'}$ is a floor state. If $F_{I'} = K_{0}$, then $\Theta_{A}(R',I';F) = 0$. Otherwise, $n_{b}(F_{I'}) = n_{b}(F)-2|I'| = n+2|I|-2|I'|$ and $n_{t}(F_{I'}) = n_{t}(F) = k$. Since $I \preceq I' \in \mathcal{L}_{n+2|I|}$, $I' \ominus I \in \mathcal{L}_{n}$ by Proposition~\ref{prop:prop_ominus}(i). Therefore, it must be $n_{t}(R') \neq n+2|I|-2|I'|$ or $n_{b}(R') \neq k$, i.e., $R' *_{v} F_{I'}$ is not a Catalan state, so $\Theta_{A}(R',I';F) = 0$. Consequently,
\begin{equation*}
\Theta_{A}(R,I;F) = \sum_{(R',I') \in \mathcal{P}} Q_{R',I'}(A) \, \Theta_{A}(R',I';F) = \sum_{(R',I') \in \mathcal{P}'} Q_{R',I'}(A) \, \Theta_{A}(R',I';F).
\end{equation*}

In the latter case, for each $(R',I') \in \mathcal{P}'$ either $F_{I'}$ is a floor state with $n_{t}(F_{I'}) \neq k$ or $n_{b}(F_{I'}) \neq n+2|I|-2|I'|$, or $F_{I'} = K_{0}$. Thus, $\Theta_{A}(R',I';F) = 0$. It follows that
\begin{equation*}
\Theta_{A}(R,I;F) = 0 = \sum_{(R',I') \in \mathcal{P}'} Q_{R',I'}(A) \, \Theta_{A}(R',I';F).
\end{equation*}

For i), given $(R',I') \in \mathcal{P}'$ we see that conditions b) and c) in the definition of $\trianglelefteq$ for $(R,I \oplus J)$ and $(R',I' \oplus J)$ are clearly satisfied. Moreover, notice that $I \oplus J \preceq I' \oplus J$ and $(I' \oplus J) \ominus (I \oplus J) = I' \ominus I \in \mathcal{L}_{n}$ by Proposition~\ref{prop:prop_ominus}(iv). Thus, condition a) in the definition of $\trianglelefteq$ also holds. Therefore, $(R,I \oplus J) \trianglelefteq (R',I' \oplus J)$. Furthermore, by \eqref{eqn:Theta_A_def} and Lemma~\ref{lem:F_J_I},
\begin{equation*}
\Theta_{A}(R,I \oplus J;F) = [[R *_{v} F_{I \oplus J}]]_{A} = [[R *_{v} (F_{J})_{I}]]_{A} = \Theta_{A}(R,I;F_{J})
\end{equation*}
and
\begin{equation*}
\Theta_{A}(R,I' \oplus J;F) = [[R *_{v} F_{I' \oplus J}]]_{A} = [[R *_{v} (F_{J})_{I'}]]_{A} = \Theta_{A}(R,I';F_{J}).
\end{equation*}
Since for any $F$ and $J \in \mathrm{Fin}(\mathbb{N})$,
\begin{equation*}
\Theta_{A}(R,I;F_{J}) = \sum_{(R',I') \in \mathcal{P}'} Q_{R',I'}(A) \, \Theta_{A}(R',I';F_{J})
\end{equation*}
by \eqref{eqn:sum_over_P_prime}, it follows that
\begin{equation*}
\Theta_{A}(R,I \oplus J;F) = \sum_{(R',I') \in \mathcal{P}'} Q_{R',I'}(A) \, \Theta_{A}(R',I' \oplus J;F).
\end{equation*}

The first part of statement ii) is obvious. For the second part we see that, by \eqref{eqn:Theta_A_def} and Lemma~\ref{lem:F_J_I},
\begin{equation*}
\Theta_{A}(R,I;M *_{v} F)
= [[R *_{v} (M *_{v} F)_{I}]]_{A} 
= [[R *_{v} M *_{v} F_{I}]]_{A}
= \Theta_{A}(R *_{v} M,I;F) 
\end{equation*}
and
\begin{equation*}
\Theta_{A}(R',I';M *_{v} F)
= [[R' *_{v} (M *_{v} F)_{I'}]]_{A}
= [[R' *_{v} M *_{v} F_{I'}]]_{A}
= \Theta_{A}(R' *_{v} M,I';F).
\end{equation*}
Since for any $F$ and a middle state $M$ with $n_{t}(M) = k$,
\begin{equation*}
\Theta_{A}(R,I;M *_{v} F) = \sum_{(R',I') \in \mathcal{P}'} Q_{R',I'}(A) \, \Theta_{A}(R',I';M *_{v} F)
\end{equation*}
by \eqref{eqn:sum_over_P_prime}, it follows that
\begin{equation*}
\Theta_{A}(R *_{v} M,I;F) = \sum_{(R',I') \in \mathcal{P}'} Q_{R',I'}(A) \, \Theta_{A}(R' *_{v} M,I';F).
\end{equation*}

The first part of statement iii) is obvious since $I = \emptyset$. For the second part, we start by showing that $[[\overline{R' *_{v} F'}]]_{A} = [[\overline{R'} *_{v} F'']]_{A}$, where $F' = (\overline{F})_{I'}$ and $F'' = F_{\overline{I'}}$. Consider two cases $n_{b}(F) \neq n$ or $n_{b}(F) = n$.

In the former case, for the left-hand side of the equality, we see that either $F' = K_{0}$ or $F'$ is a floor state with $n_{b}(F') = n_{b}(\overline{F})-2|I'| \neq n-2|I'|$. Since in both situations $\overline{R' *_{v} F'}$ is not a Catalan state, $[[\overline{R' *_{v} F'}]]_{A} = 0$.
Analogously, for the right-hand side of the equality, either
$F'' = K_{0}$ or $F''$ is a floor state with $n_{b}(F'') = n_{b}(F) - 2|\overline{I'}| \neq n-2|I'|$. Thus, in both cases $\overline{R'} *_{v} F''$ is not a Catalan state and consequently, $[[\overline{R'} *_{v} F'']]_{A} = 0$.

In the latter case, if $F' \neq K_{0}$, then
$\overline{F} = F' *_{v} \varphi_{n}^{-1}(I')$ and consequently $F = \overline{F'} *_{v} \overline{\varphi_{n}^{-1}(I')} = \overline{F'} *_{v} \varphi_{n}^{-1}(\overline{I'})$. It follows that $F'' = F_{\overline{I'}} = \overline{F'}$ and therefore $\overline{R' *_{v} F'} = \overline{R'} *_{v} \overline{F'} = \overline{R'} *_{v} F''$. If $F' = K_{0}$, we show that $F'' = K_{0}$. Suppose that $F'' \neq K_{0}$, then $F = F'' *_{v} \varphi_{n}^{-1}(\overline{I'})$ and consequently $\overline{F} = \overline{F''} *_{v} \overline{\varphi_{n}^{-1}(\overline{I'})} = \overline{F''} *_{v} \varphi_{n}^{-1}(I')$. It follows that $F' = (\overline{F})_{I'} = \overline{F''} \neq K_{0}$, a contradiction. Therefore, $\overline{R' *_{v} F'} = K_{0} = \overline{R'} *_{v} F''$. 

As we see from the above, 
\begin{equation*}
[[\overline{R' *_{v} ( 
\overline{F})_{I'}}]]_{A} = [[\overline{R' *_{v} F'}]]_{A} = [[\overline{R'} *_{v} F'']]_{A} = [[\overline{R'} *_{v} F_{\overline{I'}}]]_{A}
\end{equation*}
for all $F$. Therefore, by \eqref{eqn:Theta_A_def} and Remark~\ref{rem:prop_coef}(i),
\begin{equation*}
\Theta_{A^{-1}}(R,\emptyset;\overline{F}) 
= [[R *_{v} \overline{F}]]_{A^{-1}}
= [[\overline{\overline{R} *_{v} F}]]_{A^{-1}}
= [[\overline{R} *_{v} F]]_{A}
= \Theta_{A}(\overline{R},\emptyset;F)
\end{equation*}
and 
\begin{equation*}
\Theta_{A^{-1}}(R',I';\overline{F})
= [[R' *_{v} (\overline{F})_{I'}]]_{A^{-1}}
= [[\overline{R' *_{v} ( 
\overline{F})_{I'}}]]_{A}
= [[\overline{R'} *_{v} F_{\overline{I'}}]]_{A} 
= \Theta_{A}(\overline{R'},\overline{I'};F).
\end{equation*}
Since by \eqref{eqn:sum_over_P_prime},
\begin{equation*}
\Theta_{A^{-1}}(R,\emptyset;\overline{F}) = \sum_{(R',I') \in \mathcal{P}'} Q_{R',I'}(A^{-1}) \, \Theta_{A^{-1}}(R',I';\overline{F})
\end{equation*}
for all $F$, it follows that
\begin{equation*}
\Theta_{A}(\overline{R},\emptyset;F) = \sum_{(R',I') \in \mathcal{P}'} Q_{R',I'}(A^{-1}) \, \Theta_{A}(\overline{R'},\overline{I'};F).
\end{equation*}
\end{proof}

Consider rectangle $\mathrm{R}^{2}_{m,n,2k-n}$ as shown in Figure~\ref{fig:states_intro}(a) with $m \geq 1$. We label points $y_{1}$ by $x_{0}$ and $y'_{1}$ by $x_{n+1}$ and denote by $e_{j}$ an arc that joins $x_{j}$ and $x_{j+1}$ for $0 \leq j \leq n$ (see Figure~\ref{fig:R_J}(a)).\footnote{If $n=0$ and $m \geq 1$, $e_{0}$ is an arc with ends $y_{1},y'_{1}$.} For $\mathrm{R}^{2}_{0,n,2k-n}$, we define $e_{j}$ for $1 \leq j \leq n-1$ in an analogous way. 

Given a crossingless connection $C$, we write $e_{j} \in C$ if $C$ has an arc that is regularly isotopic to $e_{j}$ and define
\begin{equation*}
\mathcal{J}(C) = \{j \mid e_{j} \in C\}.
\end{equation*}
For example, $\mathcal{J}(R) = \{0,4,6\}$ for the roof state $R$ shown in Figure~\ref{fig:R_J}(b). 

\begin{figure}[ht] 
\centering
\includegraphics[scale=1]{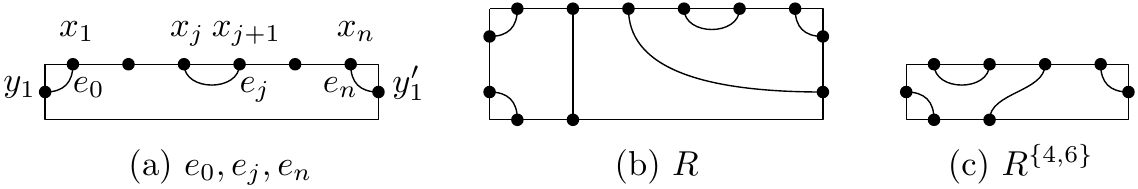}
\caption{Arcs $e_{0},e_{j},e_{n}$, roof state $R$, and $R^{J}$ for $J = \{4,6\}$}
\label{fig:R_J}
\end{figure}

Let $\mathcal{D}_{n}$ be the set of all pairs $(J,I)$ of finite subsets of $\mathbb{Z}_{\geq 0}$ such that $J = \{j_{1},j_{2},\ldots,j_{t+1}\}$ and $I = \{i_{1},i_{2},\ldots,i_{t}\}$, where $t \geq 0$ and $0 \leq j_{k} < i_{k} < j_{k+1} \leq n$ for all $k$. There is a bijection between $\mathcal{D}_{n}$ and the set $\mathcal{K}(1,n)$ of all Kauffman states of $L(1,n)$ that sends $(J,I)$ to
\begin{equation*}
s(J,I) = (\underbrace{1,1,\ldots,1}_{j_{1}},\underbrace{-1,-1,\ldots,-1}_{i_{1}-j_{1}},\underbrace{1,1,\ldots,1}_{j_{2}-i_{1}},\underbrace{-1,-1,\ldots,-1}_{i_{2}-j_{2}},\ldots,\underbrace{1,1,\ldots,1}_{j_{t+1}-i_{t}},\underbrace{-1,-1,\ldots,-1}_{n-j_{t+1}}).
\end{equation*}
Define a family of roof states
\begin{equation*}
\mathcal{R}_{n} = \{ (C_{s(J,I)})_{I} \mid (J,I) \in \mathcal{D}_{n} \}
\end{equation*}
and a map $\psi_{n}$ from $\mathcal{R}_{n}$ to the set of all subsets of $\{0,1,\ldots,n\}$ by 
\begin{equation*}
\psi_{n}(R) = \mathcal{J}(R).
\end{equation*}
Let $\mathcal{U}_{n} = \psi_{n}(\mathcal{R}_{n})$, i.e., $\mathcal{U}_{n}$ is the set of $J = \{j_{1},j_{2},\ldots,j_{t+1}\}$ such that $0 \leq j_{k} \leq n$ and $j_{k+1}-j_{k} > 1$ for all $k$.

\begin{lemma}
\label{lem:psi_injective}
$\psi_{n}$ is a bijection between $\mathcal{R}_{n}$ and $\mathcal{U}_{n}$.
\end{lemma}

\begin{proof}
The surjectivity of $\psi_{n}$ is clear, so it suffices to show that $\psi_{n}$ is injective. Suppose that $\psi_{n}(R) = \psi_{n}(R')$ for some $R = (C_{s(J,I)})_{I} \in \mathcal{R}_{n}$ and $R' = (C_{s(J',I')})_{I'} \in \mathcal{R}_{n}$. Notice that, for a fixed $j \neq 0,n$, an arc $e_{j} \in R$ if and only if the $j$-th entry of $s(J,I)$ is $1$ and the $(j+1)$-th entry of $s(J,I)$ is $-1$. Analogously, $e_{0} \in R$ if and only if the first entry of $s(J,I)$ is $-1$, and $e_{n} \in R$ if and only if the last entry of $s(J,I)$ is $1$. Thus, $\mathcal{J}(R) = J$. The same argument holds for $R'$, so $\mathcal{J}(R') = J'$. Therefore, 
\begin{equation*}
J = \mathcal{J}(R) = \psi_{n}(R) = \psi_{n}(R') = \mathcal{J}(R') = J'
\end{equation*} 
and consequently $R = (C_{s(J,I)})_{I} = (C_{s(J',I')})_{I'} = R'$.
\end{proof}

Let $C$ be a crossingless connection with $n_{t}(C) = n$ and $\mathrm{ht}(C) \geq 1$. Given $J \in \mathcal{U}_{n}$, define $C^{J}$ as the unique crossingless connection $C'$ such that 
\begin{equation*}
C = \psi_{n}^{-1}(J) *_{v} C'    
\end{equation*}
provided that $C'$ exists, and set $C^{J} = K_{0}$ otherwise. Thus, geometrically $C^{J}$ is obtained from $C$ by removing all $e_{j}$'s for $j \in J$ together with their ends and then moving left- and right-boundary points along the sides of $C$ to its top (if necessary), so that $\mathrm{ht}(C^{J}) = \mathrm{ht}(C)-1$. For example, $R^{J}$ for the roof state $R$ given in Figure~\ref{fig:R_J}(b) and $J = \{4,6\}$ is shown in Figure~\ref{fig:R_J}(c). We also let $C^{J}_{I} = (C^{J})_{I}$ to make notations shorter.

\begin{lemma}
\label{lem:1st_row_exp}
For $C \in \mathrm{Cat}(m,n)$ with $m \geq 1$, let $\mathcal{S}(C) = \{(J,I) \in \mathcal{D}_{n} \mid C^{J}_{I} \neq K_{0}\}$. Then
\begin{equation}
C(A) = \sum_{(J,I) \in \mathcal{S}(C)} A^{-n+2(\Vert{J}\Vert - \Vert{I}\Vert)} \, C^{J}_{I}(A),
\label{eqn:C(A)_new}
\end{equation}
where $\Vert{K}\Vert$ stands for the sum of (distinct) elements of $K$.
\end{lemma}

\begin{proof}
Expanding the first row of $L(m,n)$ yields
\begin{equation*}
L(m,n) = \sum_{s \in \mathcal{K}(1,n)} A^{p(s)-n(s)} \cdot C_{s} *_{v} L(m-1,n).
\end{equation*}
Since there is a bijection between $s \in \mathcal{K}(1,n)$ and $(J,I) \in \mathcal{D}_{n}$ and $\mathcal{D}_{n} \subset \mathcal{U}_{n} \times \mathcal{L}_{n}$,
\begin{equation*}
L(m,n) = \sum_{(J,I) \in \mathcal{D}_{n}} A^{-n+2(\Vert{J}\Vert - \Vert{I}\Vert)} \cdot \psi_{n}^{-1}(J) *_{v} \varphi_{n}^{-1}(I) *_{v} L(m-1,n).
\end{equation*}
Notice that diagrams $\psi_{n}^{-1}(J) *_{v} \varphi_{n}^{-1}(I) *_{v} L(m-1,n)$ and $\psi_{n}^{-1}(J) *_{v} L(m-1,n-2|I|) *_{v} \varphi_{n}^{-1}(I)$ are regularly isotopic, hence
\begin{equation*}
L(m,n) = \sum_{(J,I) \in \mathcal{D}_{n}} A^{-n+2(\Vert{J}\Vert - \Vert{I}\Vert)} \cdot \psi_{n}^{-1}(J) *_{v} L(m-1,n-2|I|) *_{v} \varphi_{n}^{-1}(I).
\end{equation*}
Furthermore, using \eqref{eqn:L_mn} for $L(m-1,n-2|I|)$ yields
\begin{equation*}
L(m,n) = \sum_{(J,I) \in \mathcal{D}_{n}} \, \sum_{C' \in \mathrm{Cat}(m-1,n-2|I|)} A^{-n+2(\Vert{J}\Vert - \Vert{I}\Vert)} \, C'(A) \cdot \psi_{n}^{-1}(J) *_{v} C' *_{v} \varphi_{n}^{-1}(I).
\end{equation*}
Notice that, if for $(J,I) \in \mathcal{D}_{n}$ and $C' \in \mathrm{Cat}(m-1,n-2|I|)$ we let $C = \psi_{n}^{-1}(J) *_{v} C' *_{v} \varphi_{n}^{-1}(I)$, then $C \in \mathrm{Cat}(m,n)$ and $(J,I) \in \mathcal{S}(C)$. Conversely, for $C \in \mathrm{Cat}(m,n)$ and $(J,I) \in \mathcal{S}(C)$, if we let $C' = C^{J}_{I}$, then $(J,I) \in \mathcal{D}_{n}$ and $C' \in \mathrm{Cat}(m-1,n-2|I|)$. Therefore
\begin{eqnarray*}
L(m,n) &=& \sum_{C \in \mathrm{Cat}(m,n)} \, \sum_{(J,I) \in \mathcal{S}(C)} A^{-n+2(\Vert{J}\Vert - \Vert{I}\Vert)} \, C^{J}_{I}(A) \cdot \psi_{n}^{-1}(J) *_{v} C^{J}_{I} *_{v} \varphi_{n}^{-1}(I) \\
&=& \sum_{C \in \mathrm{Cat}(m,n)} \Big( \sum_{(J,I) \in \mathcal{S}(C)} A^{-n+2(\Vert{J}\Vert - \Vert{I}\Vert)} \, C^{J}_{I}(A) \Big) \, C.
\end{eqnarray*}
It follows that $C(A)$ is given by \eqref{eqn:C(A)_new}.
\end{proof}

\begin{corollary} 
\label{cor:non_negative_coef}
Coefficients of $C(A)$ are non-negative integers for any Catalan state $C$.
\end{corollary}

\begin{proof}
The statement follows from \eqref{eqn:C(A)_new} by induction on $\mathrm{ht}(C)$.
\end{proof}

Given a roof state $R$ with $n_{t}(R) = n$ and $\mathrm{ht}(R) \geq 1$, define
\begin{equation*}
\mathcal{H}(R) = \{ (J,I) \in \mathcal{D}_{n} \mid J \subseteq \mathcal{J}(R) \}.
\end{equation*}

\begin{proposition} 
\label{prop:1st_row_exp}
Let $(R,\tilde{I}) \in \mathcal{W}$ with $n_{t}(R) = n$ and $\mathrm{ht}(R) \geq 1$. Then $(R,\tilde{I}) \trianglelefteq (R^{J},I \oplus \tilde{I})$ for all $(J,I) \in \mathcal{H}(R)$ and
\begin{equation}
\Theta_{A}(R,\tilde{I};\cdot) = \sum_{(J,I) \in \mathcal{H}(R)} A^{-n+2(\Vert{J}\Vert - \Vert{I}\Vert )} \, \Theta_{A}(R^{J},I \oplus \tilde{I};\cdot).
\label{eqn:1st_row_exp}
\end{equation}
\end{proposition}

We will refer to the formula \eqref{eqn:1st_row_exp} as the \emph{first-row expansion} of $(R,\tilde{I})$.

\begin{proof}
By Proposition~\ref{prop:P_prime_concat_reflect_formula}(i), it suffices to show that the above statement holds for $\tilde{I} = \emptyset$. 

Given $(J,I) \in \mathcal{H}(R)$, clearly $I \ominus \emptyset = I \in \mathcal{L}_{n}$ and $n_{b}(R^{J}) = n_{b}(R)$. Thus, conditions a) and c) in definition of $\trianglelefteq$ are satisfied. Moreover, since $(J,I) \in \mathcal{D}_{n}$, we see that $|J| = |I|+1$. Therefore,
\begin{equation*}
n_{t}(R^{J})+2|I| = n_{t}(R)+2-2|J|+2|I| = n_{t}(R),
\end{equation*}
so condition b) also holds. Hence, $(R,\emptyset) \trianglelefteq (R^{J},I)$ for all $(J,I) \in \mathcal{H}(R)$. 

Assume that $F$ is a floor state such that $C = R *_{v} F$ is a Catalan state. Since $C^{J}_{I} = R^{J} *_{v} F_{I}$ and $[[R^{J} *_{v} F_{I}]]_{A} = 0$ if $(J,I) \in \mathcal{H}(R) \setminus \mathcal{S}(C)$, where $\mathcal{S}(C)$ is defined in the statement of Lemma~\ref{lem:1st_row_exp}, by \eqref{eqn:C(A)_new}
\begin{equation*}
[[R *_{v} F]]_{A} = C(A) = \sum_{(J,I) \in \mathcal{H}(R)} A^{-n+2(\Vert{J}\Vert - \Vert{I}\Vert)} \, [[R^{J} *_{v} F_{I}]]_{A}.
\end{equation*}
If $R *_{v} F$ is not a Catalan state, as one may verify, $R^{J} *_{v} F_{I}$ is also not a Catalan state for any $(J,I) \in \mathcal{H}(R)$, hence
\begin{equation*}
[[R *_{v} F]]_{A} = 0 = \sum_{(J,I) \in \mathcal{H}(R)} A^{-n+2(\Vert{J}\Vert - \Vert{I}\Vert)} \, [[R^{J} *_{v} F_{I}]]_{A}. 
\end{equation*}
Consequently,
\begin{equation*}
\Theta_{A}(R,\emptyset;F) = \sum_{(J,I) \in \mathcal{H}(R)} A^{-n+2(\Vert{J}\Vert - \Vert{I}\Vert)} \, \Theta_{A}(R^{J},I;F)
\end{equation*}
for any floor state $F$.
\end{proof}

\begin{remark}
\label{rem:minimal_refinement}
Let $\mathcal{P} = \{(R^{J}, I \oplus \tilde{I}) \mid (J,I) \in \mathcal{H}(R)\}$, then \eqref{eqn:1st_row_exp} can be written as
\begin{equation}
\Theta_{A}(R,\tilde{I};\cdot) = \sum_{(R',I') \in \mathcal{P}} Z_{R',I'}(A) \, \Theta_{A}(R',I';\cdot),
\label{eqn:1st_row_exp_over_P}
\end{equation}
where 
\begin{equation*}
Z_{R',I'}(A) = \sum_{(J,I) \in \mathcal{A}_{R',I'}} A^{-n+2(\Vert{J}\Vert-\Vert{I}\Vert)}
\end{equation*}
and $\mathcal{A}_{R',I'} = \{(J,I) \in \mathcal{H}(R) \mid (R^{J},I \oplus \tilde{I}) = (R',I')\}$. We will also be referring to \eqref{eqn:1st_row_exp_over_P} as the first-row expansion of $(R,\tilde{I})$. 
\end{remark}

\begin{theorem} 
\label{thm:realizable_coef_non_zero}
For any Catalan state $C$, $C(A) \neq 0$ if and only if $C$ is realizable.
\end{theorem}

\begin{proof}
Let $C \in \mathrm{Cat}(m,n)$. If $C(A) \neq 0$ then by \eqref{eqn:C(A)} the set $\mathcal{K}(C)$ is nonempty, so $C$ is realizable. We show that if $C$ is realizable then $C(A) \neq 0$. By Theorem~\ref{thm:vh_line_condi}, it suffices to show that, if $C$ satisfies conditions $\#(C \cap l^{h}_{i}) \leq n$ for $i = 1,2,\ldots,m-1$ and $\#(C \cap l^{v}_{j}) \leq m$ for $j = 1,2,\ldots,n-1$, then $C(A) \neq 0$.

We prove the above statement by induction on $m$. Indeed, when $m = 0$ then clearly $C(A) = 1$. In the case $m = 1$, $\#(C \cap l^{v}_{j}) = 1$ for $j = 1,2,\ldots,n-1$, by Remark~\ref{rem:hprod}, $C$ is a horizontal product of Catalan states $C_{i} \in \mathrm{Cat}(1,1)$ and therefore $C(A) = \prod_{i=1}^{n} C_{i}(A) = A^{k}$ for some $k$. Thus $C(A)\neq 0$ in both cases. 

Let $m \geq 2$, then using Remark~\ref{rem:hprod} we may assume that $\#(C \cap l^{v}_{j}) \leq m-2$ for $j = 1,2,\ldots,n-1$. If $|\mathcal{J}(C)| = 0$ then $\#(C \cap l^{h}_{1}) = n+2$ contradicting our assumption, so $|\mathcal{J}(C)| \geq 1$. Assume that $j' \in \mathcal{J}(C)$ and let $C' = C^{J}_{I}$ for $(J,I) = (\{j'\},\emptyset)$. Since $C'$ and $C$ intersect horizontal lines the same number of times and $\#(C' \cap l^{v}_{j}) \leq \#(C \cap l^{v}_{j})+1 \leq m-1$ for $j = 1,2,\ldots,n-1$, by the induction hypothesis, $C'(A) \neq 0$. According to Lemma~\ref{lem:1st_row_exp}, 
\begin{equation*}
C(A) = \sum_{(J,I) \in \mathcal{S}(C)} A^{-n+2(\Vert{J}\Vert - \Vert{I}\Vert)} \, C^{J}_{I}(A)
\end{equation*}
and, by Corollary~\ref{cor:non_negative_coef}, coefficients of $C^{J}_{I}(A)$ are non-negative. Since $C'(A) \neq 0$ appears in the above sum, it follows that $C(A) \neq 0$.
\end{proof}

Recall, for a Catalan state $C$ with no bottom returns, $Q_{q}(T(C),v_{0},\alpha)$ denotes its plucking polynomial of the plane rooted tree.

\begin{proposition}
\label{prop:non_zero_plucking}
A Catalan state $C$ with no bottom returns is realizable if and only if $Q_{q}(T(C),v_{0},\alpha) \neq 0$.
\end{proposition}

\begin{proof}
If $C$ is realizable then $C(A) \neq 0$ by Theorem~\ref{thm:realizable_coef_non_zero} and consequently, $Q_{q}(T(C),v_{0},\alpha) \neq 0$ by Theorem~\ref{thm:coef_no_bot_rtn}. It remains to show that if $C \in \mathrm{Cat}(m,n)$ with no bottom returns is non-realizable then $Q_{q}(T(C),v_{0},\alpha) = 0$. Using Lemma~2.2 of \cite{DLP2015}, for $C$ with no bottom returns it must be $\#(C \cap l^{v}_{j}) \leq m$ for all $j = 1,2,\ldots,n-1$. Since $C$ is not realizable, there is $1 \leq i \leq m-1$ such that $\#(C \cap l^{h}_{i}) > n$ by Theorem~\ref{thm:vh_line_condi}, and in particular, $m \geq 2$.

Our proof is by induction on $m \geq 2$. For $m = 2$, since $\#(C \cap l^{h}_{1}) > n$, it must be $\#(C \cap l^{h}_{1}) = n+2$. Consequently, $T(C)$ has no leaves $v$ with $\alpha(v) = 1$, so $Q_{q}(T(C),v_{0},\alpha) = 0$ by Definition~\ref{def:q_poly}. For $m > 2$, if $\#(C \cap l^{h}_{1}) > n$ then $Q_{q}(T(C),v_{0},\alpha) = 0$ by the same argument as in case $m = 2$. Assume that $\#(C \cap l^{h}_{i}) > n$ for some $i > 1$. By Definition~\ref{def:q_poly},
\begin{equation*}
Q_{q}(T(C),v_{0},\alpha) 
= \sum_{v \in L_{1}(T(C))} q^{r(T(C),v_{0},v)} \, Q_{q}(T(C)-v,v_{0},\alpha_{v}).
\end{equation*}
For each $v \in L_{1}(T(C))$, let $C_{v} = C^{\{j_{v}\}} \in \mathrm{Cat}(m-1,n)$, where $j_{v}$ is the index of $e_{j}$ that corresponds to $v$. Clearly, $C_{v}$ is not realizable since $\#(C_{v} \cap l^{h}_{i-1}) = \#(C \cap l^{h}_{i}) > n$. Moreover, $T(C_{v}) = T(C)-v$ and $\alpha_{v}$ is its delay, so $Q_{q}(T(C)-v,v_{0},\alpha_{v}) = Q_{q}(T(C_{v}),v_{0},\alpha_{v}) = 0$ by induction hypothesis. Consequently, $Q_{q}(T(C),v_{0},\alpha) = 0$. 
\end{proof}

\begin{remark}
\label{rem:coef_no_bot_rtn_g}
From the result above, the formula for $C(A)$ given in Theorem~\ref{thm:coef_no_bot_rtn} can be used for all Catalan states $C$ with no bottom returns if we set $\beta(C) = 0$ when $C$ is non-realizable and define $Q^{*}(T(C),v_{0},\alpha) = 0$ whenever $Q(T(C),v_{0},\alpha) = 0$.
\end{remark}

\section{\texorpdfstring{$\Theta_{A}$}{\unichar{"03F4}}-state Expansion}
\label{s:Theta_state_expansion}
In this section we show that coefficient of any Catalan state $C$ with a given roof state $R$ (i.e., $C = R *_{v} F$ for some floor state $F$) can be found using coefficients of Catalan states with no top returns. For this purpose, we introduce a $\Theta_{A}$-state expansion for $(R,I)$ (see Definition~\ref{def:Theta_state_expansion}) that expresses $\Theta_{A}(R,I;\cdot)$ as a linear combination (over the field of rational functions) of $\Theta_{A}(R',I';\cdot)$'s, where $R'$ is a middle state. The coefficient of $C$ can be found by evaluating $\Theta_{A}$-state expansion for $(R,\emptyset)$ at $F$. Since each $R' *_{v} F_{I'}$ is either a Catalan state with no top returns or $K_{0}$, $\Theta_{A}(R',I';F) = [[R' *_{v} F_{I'}]]_{A}$ can be found by either applying Theorem~\ref{thm:coef_no_bot_rtn} or simply $\Theta_{A}(R',I';F) = 0$. As we will prove,  $\Theta_{A}$-state expansion exists for every $(R,I)$ (Theorem~\ref{thm:main}) and the proof of Theorem~\ref{thm:main} yields an algorithm (Algorithm~\ref{alg:Theta_state_expansion}) for finding a $\Theta_{A}$-state expansion for any pair $(R,I)$. This gives us an efficient method for finding coefficients of Catalan states with any given roof state. 

\begin{definition}
\label{def:Theta_state_expansion}
A \emph{$\Theta_{A}$-state expansion} for $(R,I) \in \mathcal{W}$ is a relation given by
\begin{equation}
\Theta_{A}(R,I;\cdot) = \sum_{(R',I') \in \mathcal{P}'} Q_{R',I'}(A) \, \Theta_{A}(R',I';\cdot),
\label{eqn:Theta_state_expansion}
\end{equation}
where $\mathcal{P}'$ is a finite subset of $\mathcal{W}$ that satisfies conditions
\begin{enumerate}
\item[i)] $R'$ in each pair $(R',I') \in \mathcal{P}'$ is a middle state,
\item[ii)] $\mathcal{P}' \trianglerighteq (R,I)$,
\end{enumerate}
and $0 \neq Q_{R',I'}(A) \in \mathbb{Q}(A)$ for every $(R',I') \in \mathcal{P}'$.
\end{definition}

\begin{figure}[ht] 
\centering
\includegraphics[scale=1]{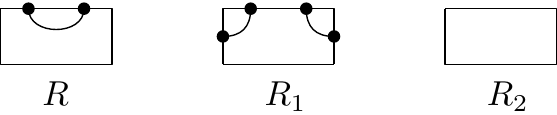}
\caption{Roof states $R$, $R_{1}$, and $R_{2}$}
\label{fig:ex_Theta_state_expansion_1}
\end{figure}

\begin{example}
\label{ex:Theta_state_expansion_1}
A $\Theta_{A}$-state expansion for $(R,\emptyset)$, where $R$ is shown in Figure~\ref{fig:ex_Theta_state_expansion_1}, is given by
\begin{equation*}
\Theta_{A}(R,\emptyset;\cdot)
= \frac{1}{A^{-2}+A^{2}} \, \Theta_{A}(R_{1},\emptyset;\cdot) 
- \frac{1}{A^{-2}+A^{2}} \, \Theta_{A}(R_{2},\{1\};\cdot),
\end{equation*}
where $R_{1}$ and $R_{2}$ are shown in Figure~\ref{fig:ex_Theta_state_expansion_1}. The equation above can be obtained using the first-row expansion \eqref{eqn:1st_row_exp} for $(R_{1},\emptyset)$. Moreover, for any realizable $C \in \mathrm{Cat}(m,2)$ with no top and bottom returns, by Theorem~\ref{thm:split_cat},
\begin{equation}
\label{eqn:ex_Theta_state_expansion_1}
\Theta_{A}(R,\emptyset;\cdot)
= \frac{1}{(A^{-2}+A^{2})C(A)} \, \Theta_{A}(C *_{v} R_{1},\emptyset;\cdot) 
- \frac{1}{A^{-2}+A^{2}} \, \Theta_{A}(R_{2},\{1\};\cdot).
\end{equation}
As one may verify, \eqref{eqn:ex_Theta_state_expansion_1} is a $\Theta_{A}$-state expansion for $(R,\emptyset)$. Thus, in fact there are infinitely many different $\Theta_{A}$-state expansions for $(R,\emptyset)$.
\end{example}

\begin{figure}[ht] 
\centering
\includegraphics[scale=1]{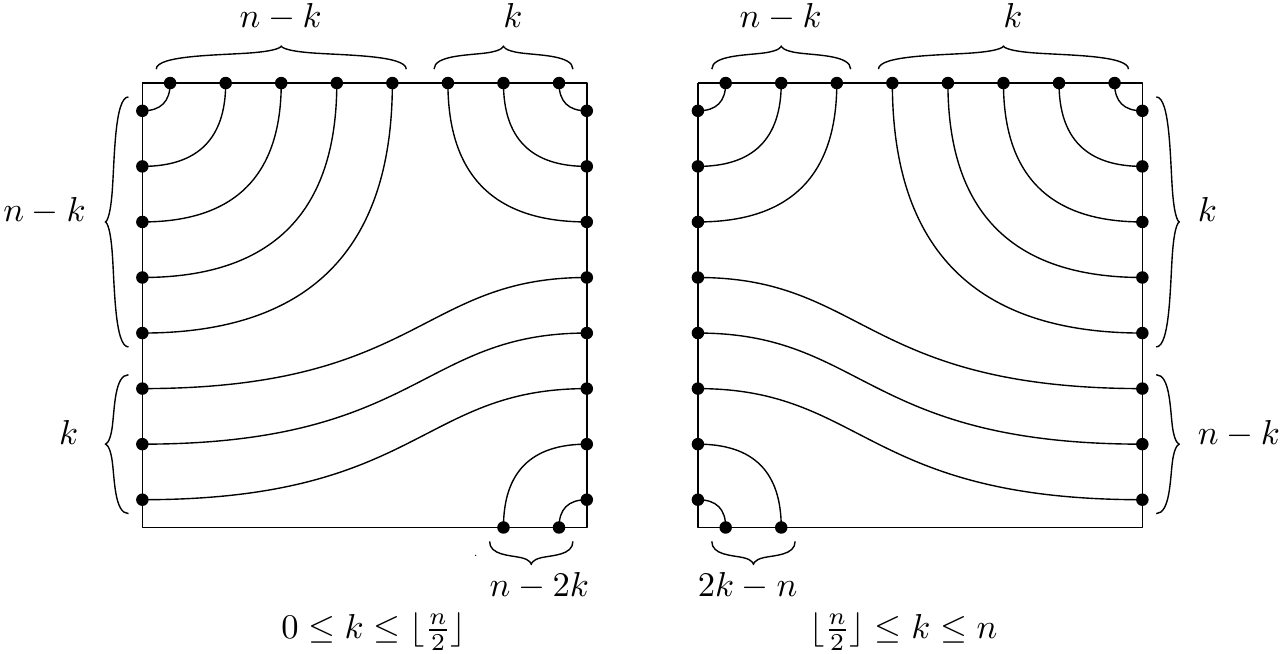}
\caption{Middle states $R_{n,k}$}
\label{fig:lem_roof_R_nk_1}
\end{figure}

\begin{lemma}
\label{lem:roof_R_nk_1}
Let $R_{n,k}$ and $R'_{n,k,l}$, $0 \leq k \leq n$ and $0 \leq l \leq \min\{k,n-k\}$, be as in Figure~\ref{fig:lem_roof_R_nk_1} and Figure~\ref{fig:lem_roof_R_nk_2}, respectively. Then there are Laurent polynomials $Z_{R,I}(A)$ such that
\begin{equation}
\Theta_{A}(R_{n,k},\emptyset;\cdot) = \sum_{(R,I) \in \mathcal{P}} Z_{R,I}(A) \, \Theta_{A}(R,I;\cdot), 
\label{eqn:roof_R_nk}
\end{equation}
where $\mathcal{P} = \{(R'_{n,k,|I|},I) \mid I \in \mathcal{L}_{n}, \, |I| \leq \min\{k,n-k\}\} \trianglerighteq (R_{n,k},\emptyset)$.
\end{lemma}

\begin{figure}[ht] 
\centering
\includegraphics[scale=1]{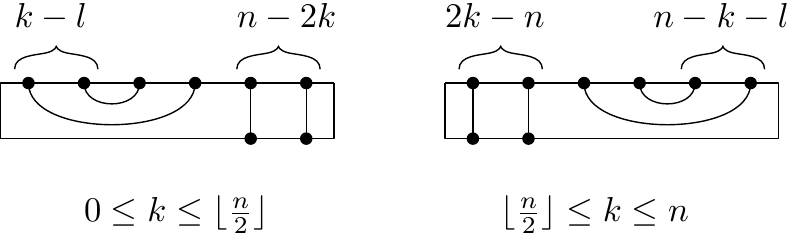}
\caption{Top states $R'_{n,k,l}$}
\label{fig:lem_roof_R_nk_2}
\end{figure}

\begin{proof}
For $n = 0$ the statement is clearly true, so we assume that $n = \mathrm{ht}(R_{n,k}) \geq 1$. Using the first-row expansion \eqref{eqn:1st_row_exp_over_P} for $(R_{n,k},\emptyset)$ we see that
\begin{equation}
\Theta_{A}(R_{n,k},\emptyset;\cdot) = \sum_{(R',I') \in \mathcal{P}_{1}} Z_{R',I'}(A) \, \Theta_{A}(R',I';\cdot),
\label{eqn:pf_lem_roof_R_nk_1}
\end{equation}
where $\mathcal{P}_{1} = \{((R_{n,k})^{J}, I) \mid (J,I) \in \mathcal{H}(R_{n,k})\}$ and $Z_{R',I'}(A)$'s are Laurent polynomials.
We see that, for each $(R',I') \in \mathcal{P}_{1}$, $\mathrm{ht}(R') = n-1$. If $n \geq 2$, we may apply \eqref{eqn:1st_row_exp_over_P} to each $(R',I') \in \mathcal{P}_{1}$ and therefore $\Theta_{A}(R',I';\cdot)$ can be written as a sum over a set $\mathcal{P}_{R',I'}$. Notice that, since $\mathcal{P}_{R',I'} \trianglerighteq (R',I')$ and $\mathcal{P}_{1} \trianglerighteq (R_{n,k},\emptyset)$ by Proposition~\ref{prop:1st_row_exp}, it follows by Lemma~\ref{lem:transitivity_triangle} that $\mathcal{P}_{R',I'} \trianglerighteq (R_{n,k},\emptyset)$ for every $(R',I') \in \mathcal{P}_{1}$. After substituting for $\Theta_{A}(R',I';\cdot)$'s into \eqref{eqn:pf_lem_roof_R_nk_1} and collecting like terms, we see that
\begin{equation*}
\Theta_{A}(R_{n,k},\emptyset;\cdot) = \sum_{(R'',I'') \in \mathcal{P}_{2}} Z_{R'',I''}(A) \, \Theta_{A}(R'',I'';\cdot),
\end{equation*}
where $\mathcal{P}_{2} = \bigcup_{(R',I') \in \mathcal{P}_{1}} \mathcal{P}_{R',I'}$. Clearly, $\mathcal{P}_{2} \trianglerighteq (R_{n,k},\emptyset)$, $\mathrm{ht}(R'') = 0$ for each $(R'',I'') \in \mathcal{P}_{2}$, and each $Z_{R'',I''}(A)$ is a Laurent polynomial. A finite number of iterations of the above procedure results in
\begin{equation}
\label{eqn:pf_lem_roof_R_nk_2}
\Theta_{A}(R_{n,k},\emptyset;\cdot) = \sum_{(R,I) \in \mathcal{P}_{n}} Z_{R,I}(A) \, \Theta_{A}(R,I;\cdot),
\end{equation}
where $\mathcal{P}_{n} \trianglerighteq (R_{n,k},\emptyset)$, $\mathrm{ht}(R) = 0$ for each $(R,I) \in \mathcal{P}_{n}$, and $Z_{R,I}(A)$'s are Laurent polynomials. 

We first show that $\mathcal{P}_{n} \subseteq \mathcal{P}$, i.e., we prove that for each $(R,I) \in \mathcal{P}_{n}$, $R = R'_{n,k,|I|}$, $I \in \mathcal{L}_{n}$, and $|I| \leq \min\{k,n-k\}$. Since $\mathcal{P}_{n} \trianglerighteq (R_{n,k},\emptyset)$, we see that $I \in \mathcal{L}_{n}$ for all $(R,I) \in \mathcal{P}_{n}$. Moreover, if $R = R'_{n,k,|I|}$ then $|I| \leq \min\{k,n-k\}$. Hence, to complete our argument, it suffices to show that $R = R'_{n,k,|I|}$.

Let $(\tilde{R},\tilde{I}) \in \mathcal{P}_{i}$, $1 \leq i \leq n-1$. Since $R_{n,k}$ has neither left nor right returns, $\tilde{R}$ obtained from $R_{n,k}$ by applying first-row expansions has also neither left nor right returns and consequently $|\mathcal{J}(\tilde{R})| \leq |\mathcal{J}(R_{n,k})| = 2$. Moreover, if $|\mathcal{J}(\tilde{R})| = 0$ then each arc of $\tilde{R}$ must have one of its ends on the bottom side of $\tilde{R}$. Since $\mathrm{ht}(\tilde{R}) = n-i > 0$, 
$\tilde{R}$ has two arcs $\tilde{c}$ and $\tilde{c}'$ with ends $y_{1}$ and $y'_{1}$, respectively. However, since $\tilde{R}$ is obtained by the first-row expansions of $(R_{n,k},\emptyset)$, arcs $\tilde{c}$ and $\tilde{c}'$ should correspond to arcs joining bottom-boundary points with left- and right-boundary points of $R_{n,k}$, respectively. Since $R_{n,k}$ does not have two arcs with such property, $|\mathcal{J}(\tilde{R})| = 0$ is not possible. It follows that $|\mathcal{J}(\tilde{R})| = 1$ or $2$ for each $(\tilde{R},\tilde{I}) \in \mathcal{P}_{i}$. In the former case, let $\mathcal{J}(\tilde{R}) = \{j\}$ and then by \eqref{eqn:1st_row_exp},
\begin{equation*}
\Theta_{A}(\tilde{R},\tilde{I};\cdot) = A^{-n_{t}(\tilde{R})+2j} \, \Theta_{A}(\tilde{R}',\tilde{I}';\cdot),
\end{equation*}
where $\tilde{R}' = \tilde{R}^{\{j\}}$, $n_{t}(\tilde{R}') = n_{t}(\tilde{R})$, $\mathcal{J}(\tilde{R}') = \{j\}$ or $\emptyset$, and $\tilde{I}' = \tilde{I}$.

In the latter case, let $\mathcal{J}(\tilde{R}) = \{j_{1} < j_{2}\}$ and then by \eqref{eqn:1st_row_exp},
\begin{equation*}
\Theta_{A}(\tilde{R},\tilde{I};\cdot) = A^{-n_{t}(\tilde{R})+2j_{1}} \, \Theta_{A}(\tilde{R}'_{1},\tilde{I}'_{1};\cdot) + A^{-n_{t}(\tilde{R})+2j_{2}} \, \Theta_{A}(\tilde{R}'_{2},\tilde{I}'_{2};\cdot) + \sum_{j_{1} < a < j_{2}} A^{-n_{t}(\tilde{R})+2(j_{1}+j_{2}-a)} \, \Theta_{A}(\tilde{R}''_{a},\tilde{I}''_{a};\cdot),
\end{equation*}
where in the above sum
\begin{itemize}
\item $\tilde{R}'_{1} = \tilde{R}^{\{j_{1}\}}$, $n_{t}(\tilde{R}'_{1}) = n_{t}(\tilde{R})$, $\mathcal{J}(\tilde{R}'_{1}) = \{j_{1},j_{2}-1\}$ or $\{j_{2}-1\}$, and $\tilde{I}'_{1} = \tilde{I}$;
\item $\tilde{R}'_{2} = \tilde{R}^{\{j_{2}\}}$,  $n_{t}(\tilde{R}'_{2}) = n_{t}(\tilde{R})$, $\mathcal{J}(\tilde{R}'_{2}) = \{j_{1}+1,j_{2}\}$ or $\{j_{1}+1\}$, and $\tilde{I}'_{2} = \tilde{I}$;
\item $\tilde{R}''_{a} = \tilde{R}^{\{j_{1},j_{2}\}}$, $n_{t}(\tilde{R}''_{a}) = n_{t}(\tilde{R})-2$, $\mathcal{J}(\tilde{R}''_{a}) = \{j_{1},j_{2}-2\}$, $\{j_{1}\}$, $\{j_{2}-2\}$, or $\emptyset$, and $\tilde{I}''_{a} = \{a\} \oplus \tilde{I}$;
\end{itemize}
as shown in Figure~\ref{fig:pf_lem_roof_R_nk_1}. 

\begin{figure}[ht]
\centering
\includegraphics[scale=1]{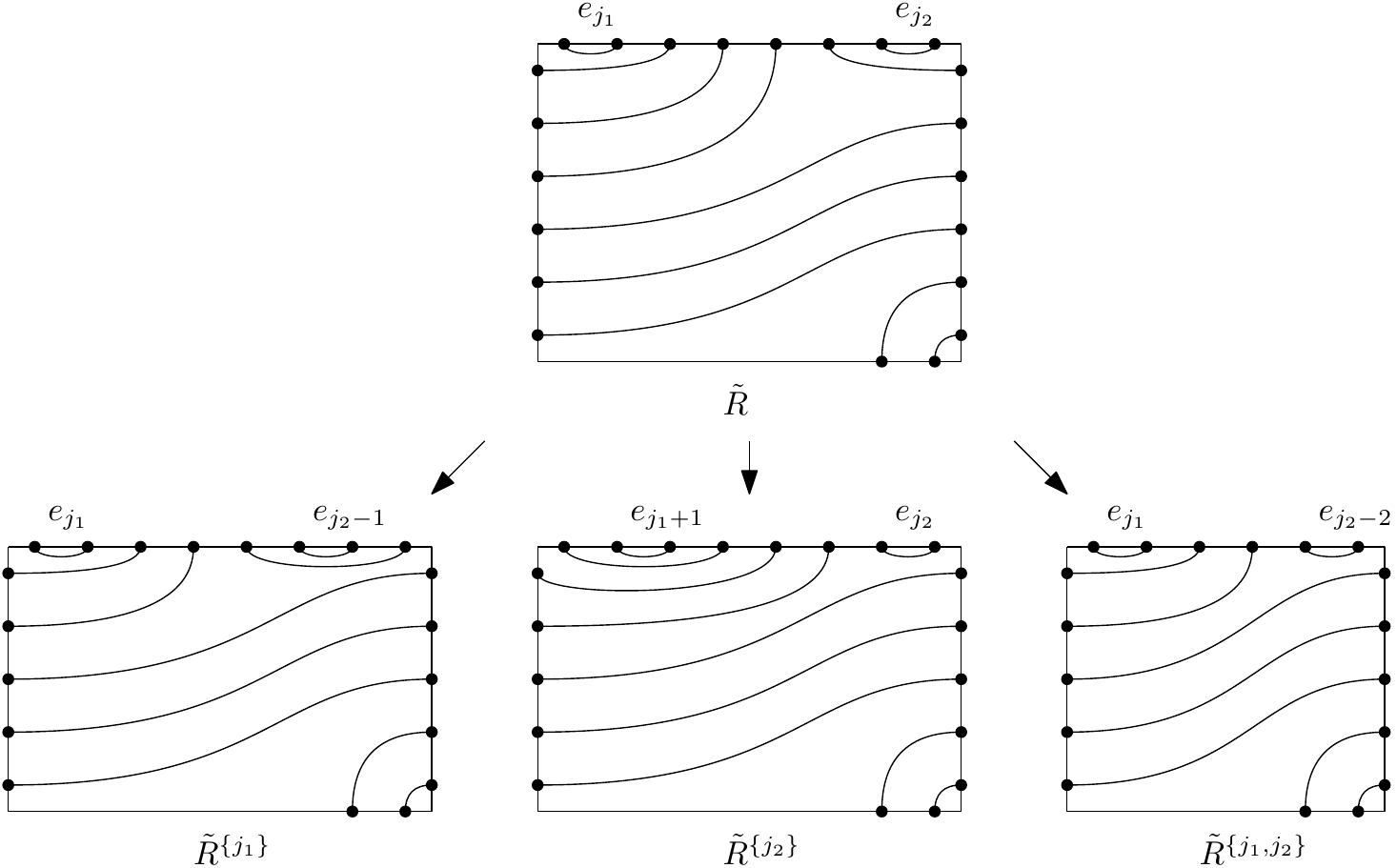}
\caption{Case $\mathcal{J}(\tilde{R}) = \{j_{1}, j_{2}\}$}
\label{fig:pf_lem_roof_R_nk_1}
\end{figure}

Since $(R,I) \in \mathcal{P}_{n}$ is obtained by the first-row expansion of some $(\tilde{R},\tilde{I}) \in \mathcal{P}_{n-1}$, either $|\mathcal{J}(R)| = 0$, $1$, or $2$. We show that $|\mathcal{J}(R)| = 2$ is not possible. Indeed, if $\mathcal{J}(R) = \{j_{1} < j_{2}\}$, then $(R,I)$ is a result of $n$ first-row expansions of pairs $(\tilde{R},\tilde{I})$ with $|\mathcal{J}(\tilde{R})| = 2$. As we argued above, in such a case each $(\tilde{R},\tilde{I}) = (\tilde{R}'_{1},\tilde{I}'_{1})$, $(\tilde{R}'_{2},\tilde{I}'_{2})$, or $(\tilde{R}''_{a},\tilde{I}''_{a})$ for some $a$.
Let $c_{1}$ be the number of times $(\tilde{R},\tilde{I}) = (\tilde{R}'_{1},\tilde{I}'_{1})$, $c_{2}$ be the number of times $(\tilde{R},\tilde{I}) = (\tilde{R}'_{2},\tilde{I}'_{2})$, and $c_{3}$ be the number of times $(\tilde{R},\tilde{I}) = (\tilde{R}''_{a},\tilde{I}''_{a})$ for some $a$. Clearly, $c_{1}+c_{2}+c_{3} = n$ and, as one can see by analyzing these three cases, $|I| = c_{3}$, $j_{1} = c_{2}$, and $j_{2} = n-c_{1}-2c_{3}$. Therefore,
\begin{equation*}
1 \leq j_{2}-j_{1} = -|I| \leq 0,
\end{equation*}
a contradiction.

Therefore, $|\mathcal{J}(R)| \leq 1$ and we see that $\mathrm{ht}(R) = 0$, $n_{b}(R) = n_{b}(R_{n,k}) = |n-2k|$, and all top returns of $R$ are to the left (respectively right) of all arcs with one end on the bottom if $0 \leq k \leq \lfloor\frac{n}{2}\rfloor$ (respectively $\lfloor\frac{n}{2}\rfloor \leq k \leq n$). It follows that $R = R'_{n,k,l}$ for some $l$ and, since $n_{t}(R) = n-2|I|$ and $n_{t}(R'_{n,k,l}) = n-2l$, it must be $l = |I|$. 

We now show that $\mathcal{P} \subseteq \mathcal{P}_{n}$. Let $(R'_{n,k,l},I) \in \mathcal{P}$, where $I = \{i_{1},i_{2},\ldots,i_{l}\} \in \mathcal{L}_{n}$ and $l = |I| \leq \min\{k,n-k\}$, and let $(R_{0},I_{0}) = (R_{n,k},\emptyset)$. It suffices to find a sequence of pairs $(R_{j+1},I_{j+1}) \in \mathcal{P}_{j+1}$, $0 \leq j \leq n-1$, such that $(R_{n},I_{n}) = (R'_{n,k,l},I)$ and $(R_{j+1},I_{j+1}) = ((R_{j})^{J'_{j}},I'_{j} \oplus I_{j})$ for some $(J'_{j},I'_{j}) \in \mathcal{H}(R_{j})$. 

If $0 \leq k \leq \lfloor\frac{n}{2}\rfloor$, let
\begin{equation*}
(J'_{j},I'_{j}) =
\begin{cases}
(\{0,n-2j\},\{i_{l-j}\}), & \text{for}\ 0 \leq j \leq l-1, \\ 
(\{n-2l\},\emptyset), & \text{for}\ l \leq j \leq k-1, \\ 
(\{k-l\},\emptyset), & \text{for}\ k \leq j \leq n-1.
\end{cases}
\end{equation*}
We argue that $(J'_{j},I'_{j}) \in \mathcal{H}(R_{j})$. For $0 \leq j \leq l-1$, since $R_{j}$ is as shown in Figure~\ref{fig:cases_R_j_1} and $\{i_{l-j}\} \preceq \{i_{1},\ldots,i_{l-j-1}\} \oplus \{i_{l-j}\} = \{i_{1},\ldots,i_{l-j}\} = I \ominus \{i_{l+1-j},\ldots,i_{l}\} \in L_{n-2j}$ by Proposition~\ref{prop:prop_oplus}(i) and Proposition~\ref{prop:prop_ominus}(i), it follows that $\mathcal{J}(R_{j}) = \{0,n-2j\}$ and $0 < i_{l-j} < n-2j$. Therefore, $(\{0,n-2j\},\{i_{l-j}\}) \in \mathcal{D}_{n-2j}$ and $\{0,n-2j\} \subseteq \mathcal{J}(R_{j})$, i.e., $(J'_{j},I'_{j}) = (\{0,n-2j\},\{i_{l-j}\}) \in \mathcal{H}(R_{j})$ for all $0 \leq j < l$. 

\begin{figure}[ht] 
\centering
\includegraphics[scale=1]{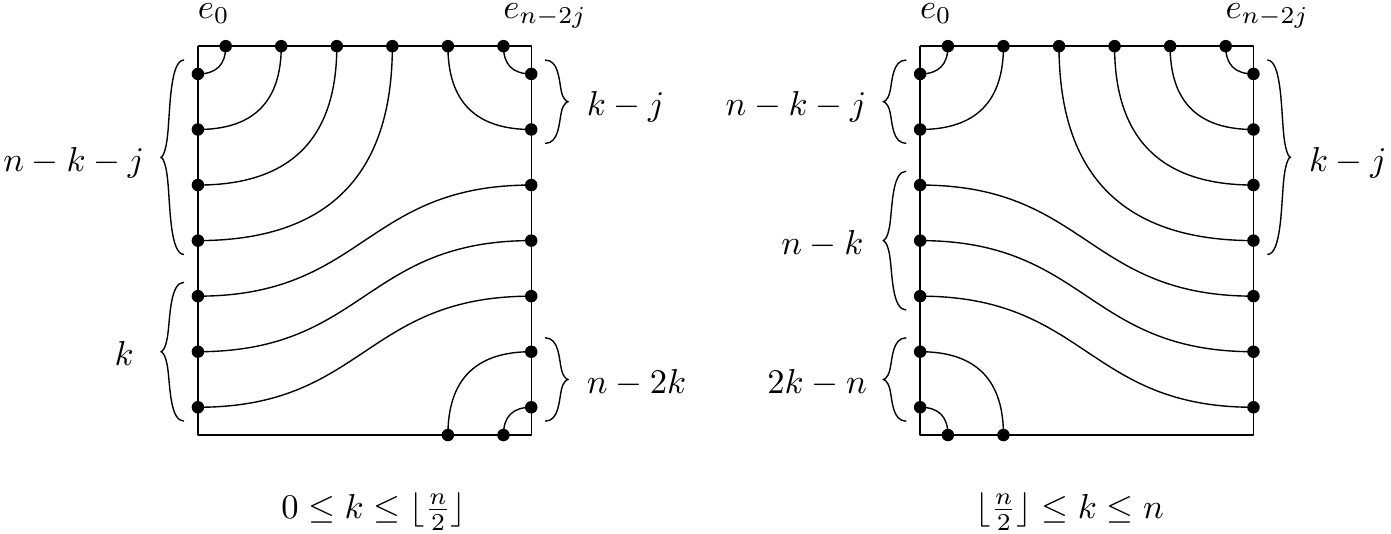}
\caption{Roof states $R_{j}$ for $0 \leq j \leq l$}
\label{fig:cases_R_j_1}
\end{figure}

After first $l$ steps $R_{l}$ is as shown in Figure~\ref{fig:cases_R_j_1} with $j = l$ and, as one may verify, $R_{j}$ is the roof state shown in Figure~\ref{fig:cases_R_j_2} for $l < j < k$. Therefore, $\mathcal{J}(R_{j}) = \{j-l,n-2l\}$ and consequently $(J'_{j},I'_{j}) = (\{n-2l\},\emptyset) \in \mathcal{H}(R_{j})$ for all $l \leq j < k$. 

\begin{figure}[ht] 
\centering
\includegraphics[scale=1]{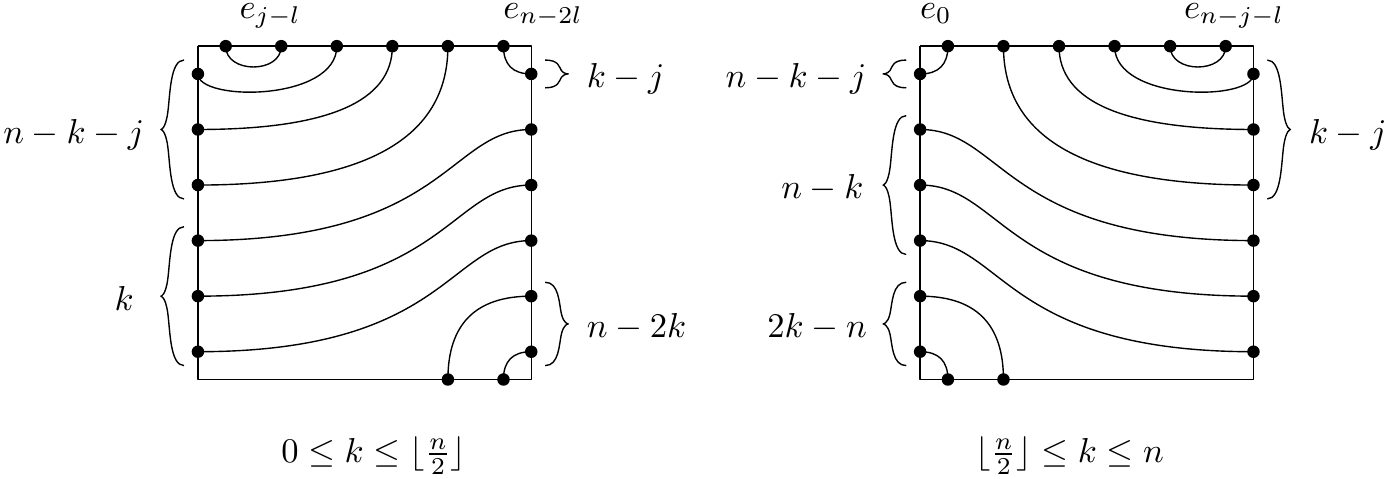}
\caption{Roof states $R_{j}$ for $l < j < \min\{k,n-k\}$}
\label{fig:cases_R_j_2}
\end{figure}

Finally, we see that $R_{j}$ is the roof state shown in Figure~\ref{fig:cases_R_j_3} for $k \leq j < n$. Therefore, $\mathcal{J}(R_{j}) = \{k-l\}$ and consequently $(J'_{j},I'_{j}) = (\{k-l\},\emptyset) \in \mathcal{H}(R_{j})$ for all $k \leq j \leq n-1$. Clearly $R_{n} = R'_{n,k,l}$ is as shown in Figure~\ref{fig:cases_R_j_3} with $j = n$ and
\begin{equation*}
I_{n} = \underbrace{\emptyset \oplus \cdots \oplus \emptyset}_{n-l} \oplus \{i_{1}\} \oplus \{i_{2}\} \oplus \cdots \oplus \{i_{l}\}  = I
\end{equation*}
by Proposition~\ref{prop:prop_oplus}(i).

\begin{figure}[ht] 
\centering
\includegraphics[scale=1]{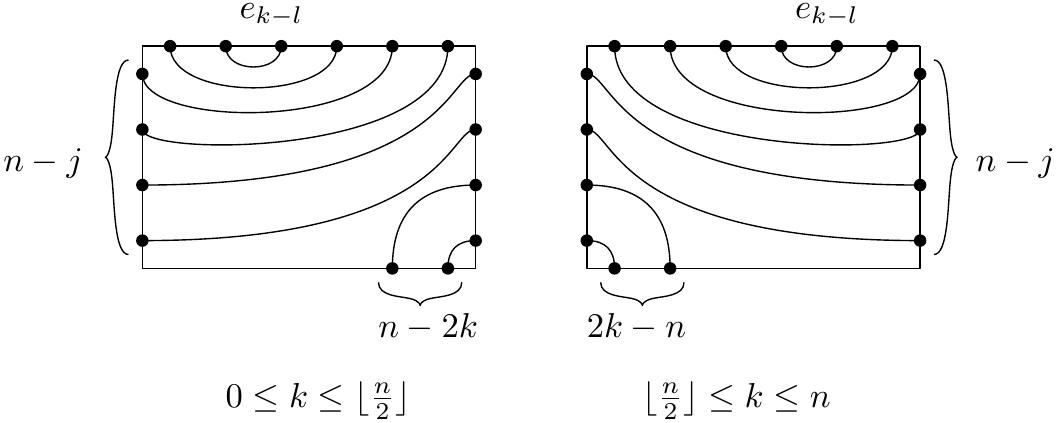}
\caption{Roof states $R_{j}$ for $\min\{k,n-k\} \leq j \leq n$}
\label{fig:cases_R_j_3}
\end{figure}

If $\lfloor\frac{n}{2}\rfloor \leq k \leq n$, let
\begin{equation*}
(J'_{j},I'_{j}) =
\begin{cases}
(\{0,n-2j\},\{i_{l-j}\}), & \text{for}\ 0 \leq j \leq l-1, \\ 
(\{0\},\emptyset), & \text{for}\ l \leq j \leq n-k-1, \\ 
(\{k-l\},\emptyset), & \text{for}\ n-k \leq j \leq n-1.
\end{cases}
\end{equation*}
Using analogous arguments as in the previous case with diagrams on the right of Figures~\ref{fig:cases_R_j_1}-\ref{fig:cases_R_j_3}, we can show that $(J'_{j},I'_{j}) \in \mathcal{H}(R_{j})$ for all $0 \leq j \leq n-1$, $R_{n} = R'_{n,k,l}$, and $I_{n} = I$.

Since as it was shown above $\mathcal{P}_{n} = \mathcal{P}$ in \eqref{eqn:pf_lem_roof_R_nk_2}, the proof is completed.
\end{proof}

\begin{lemma}
\label{lem:roof_R_nk_2}
Coefficients $Z_{R,I}(A)$ in Lemma~\ref{lem:roof_R_nk_1} are given by
\begin{equation*}
Z_{R,I}(A) = Z_{R'_{n,k,|I|},I}(A) = A^{(n-2k)(\min\{k,n-k\}-|I|)} \, [[R''_{n,k,|I|} *_{v} \varphi_{n}^{-1}(I)]]_{A},
\end{equation*}
where $R''_{n,k,l}$ are shown in Figure~\ref{fig:lem_roof_R_nk_3}. In particular, $Z_{R,I}(A) \neq 0$.
\end{lemma}

\begin{figure}[ht] 
\centering
\includegraphics[scale=1]{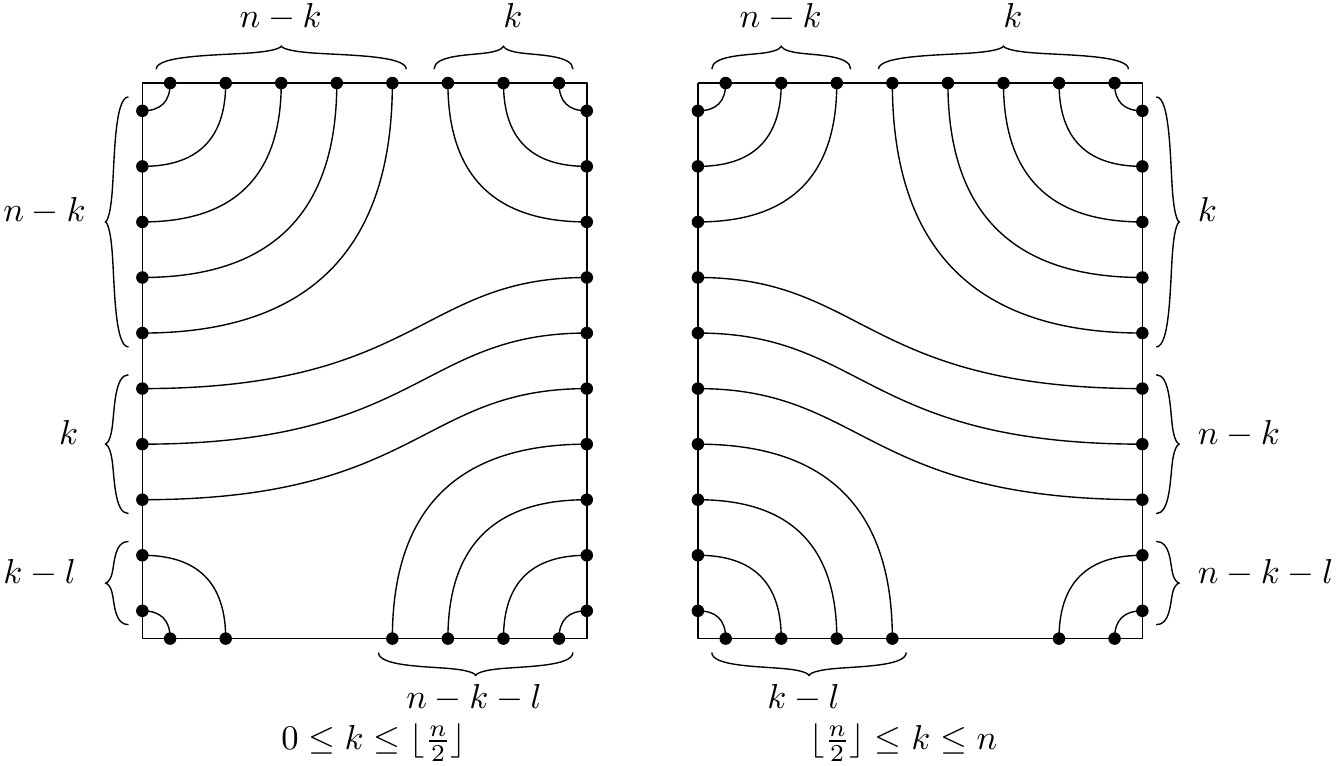}
\caption{Middle states $R''_{n,k,l}$}
\label{fig:lem_roof_R_nk_3}
\end{figure}

\begin{proof}
Given $(R'_{n,k,|I|},I) \in \mathcal{P}$, let $F'_{n,k,I} = M_{n-2|I|,k'-|I|} *_{v} \varphi_{n}^{-1}(I)$, where $M_{n,k}$ is a middle state shown in Figure~\ref{fig:pf_lem_roof_R_nk_2}(a) and $k' = \min\{k,n-k\}$. For each $I' \in \mathcal{L}_{n}$ with $I' \prec I$, define Catalan state $C = R'_{n,k,|I'|} *_{v} (F'_{n,k,I})_{I'}$. Since
\begin{equation*}
\#(C \cap l^{v}_{k-|I'|}) \geq k'-|I'| > k'-|I| = \mathrm{ht}(C) 
\end{equation*}
(see Figure~\ref{fig:pf_lem_roof_R_nk_2}(b) for case $0 \leq k \leq \lfloor \frac{n}{2}\rfloor$), by Theorem~\ref{thm:vh_line_condi}, $C$ is not realizable and consequently 
\begin{equation*}
\Theta_{A}(R'_{n,k,|I'|},I';F'_{n,k,I}) = C(A) = 0.
\end{equation*}
For each $I' \in \mathcal{L}_{n}$ with $I \prec I'$, we see that $(F'_{n,k,I})_{I'} = K_{0}$, so $\Theta_{A}(R'_{n,k,|I'|},I';F'_{n,k,I}) = 0$.
Therefore, formula \eqref{eqn:roof_R_nk} evaluated at the floor state $F'_{n,k,I}$ reduces into
\begin{equation*}
\Theta_{A}(R_{n,k},\emptyset;F'_{n,k,I}) = Z_{R,I}(A) \, \Theta_{A}(R'_{n,k,|I|},I;F'_{n,k,I}).
\end{equation*}
Since
\begin{equation*}
\Theta_{A}(R_{n,k},\emptyset;F'_{n,k,I}) = [[R_{n,k} *_{v} M_{n-2|I|,k'-|I|} *_{v} \varphi _{n}^{-1}(I)]]_{A} = [[R''_{n,k,|I|} *_{v} \varphi_{n}^{-1}(I)]]_{A}
\end{equation*} 
and, as one may verify,
\begin{equation*}
\Theta_{A}(R'_{n,k,|I|},I;F'_{n,k,I}) = [[R'_{n,k,|I|} *_{v} M_{n-2|I|,k'-|I|}]]_{A} = A^{-(n-2k)(k'-|I|)},
\end{equation*} 
it follows that $Z_{R,I}(A)$ is as we claimed.

Since the Catalan state $C' = R''_{n,k,|I|} *_{v} \varphi_{n}^{-1}(I)$ has no returns on three sides, by Lemma~2.2 of \cite{DLP2015}, $C'$ is realizable. Consequently, by Theorem~\ref{thm:realizable_coef_non_zero}, $C'(A) \neq 0$, i.e., $Z_{R,I}(A) \neq 0$.
\end{proof}

\begin{figure}[ht]
\centering
\includegraphics[scale=1]{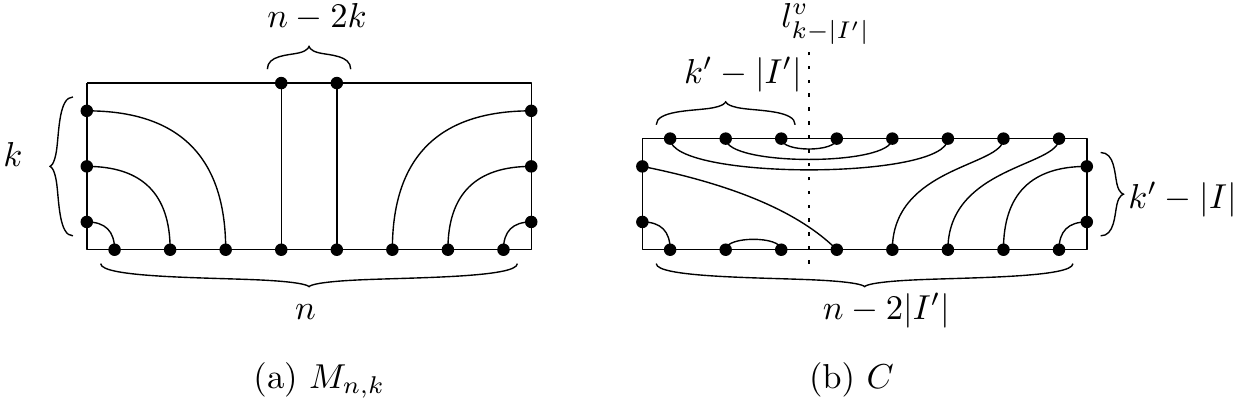}
\caption{Middle state $M_{n,k}$ and Catalan state $C$}
\label{fig:pf_lem_roof_R_nk_2}
\end{figure}

\begin{figure}[ht]
\centering
\includegraphics[scale=1]{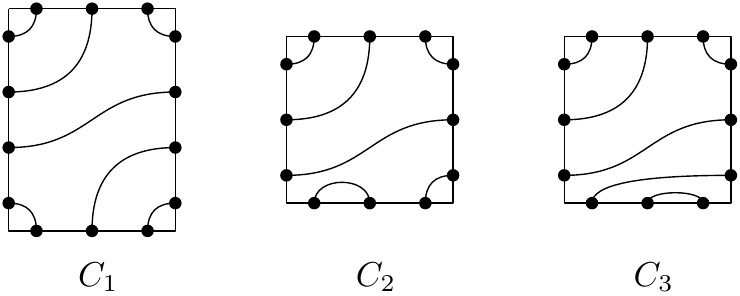}
\caption{Catalan states $C_{1}$, $C_{2}$, and $C_{3}$}
\label{fig:ex_roof_nk_n3}
\end{figure}

\begin{example}
\label{ex:roof_nk_n3}
We find \eqref{eqn:roof_R_nk} for $n = 3$ and $k = 1,2$. When $n = 3$ and $k = 1$, since $\mathcal{L}_{3} = \{\emptyset,\{1\},\{2\}\}$ and $\min\{k,n-k\} = 1$, 
\begin{equation*}
\Theta_{A}(R_{3,1},\emptyset;\cdot) = Z_{R'_{3,1,0},\emptyset}(A) \, \Theta_{A}(R'_{3,1,0},\emptyset;\cdot) + Z_{R'_{3,1,1},\{1\}}(A) \, \Theta_{A}(R'_{3,1,1},\{1\};\cdot) + Z_{R'_{3,1,1},\{2\}}(A) \, \Theta_{A}(R'_{3,1,1},\{2\};\cdot),
\end{equation*}
where 
\begin{equation*}
Z_{R'_{3,1,0},\emptyset}(A) = A^{(3-2)(1-0)} \, C_{1}(A) = A^{-7}+A^{-3}+A,
\end{equation*}
\begin{equation*}
Z_{R'_{3,1,1},\{1\}}(A) = A^{(3-2)(1-1)} \, C_{2}(A) = A^{-5}+A^{-1},
\end{equation*}
\begin{equation*}
Z_{R'_{3,1,1},\{2\}}(A) = A^{(3-2)(1-1)} \, C_{3}(A) =  A^{-3},
\end{equation*}
and coefficients of Catalan states $C_{1}, C_{2}$, and $C_{3}$ shown in Figure~\ref{fig:ex_roof_nk_n3} are computed using Theorem~\ref{thm:coef_no_bot_rtn}.

Analogously, equation \eqref{eqn:roof_R_nk} for $n = 3$ and $k = 2$ is
\begin{equation*}
\Theta_{A}(R_{3,2},\emptyset;\cdot) = Z_{R'_{3,2,0},\emptyset}(A) \, \Theta_{A}(R'_{3,2,0},\emptyset;\cdot) + Z_{R'_{3,2,1},\{1\}}(A) \, \Theta_{A}(R'_{3,2,1},\{1\};\cdot) + Z_{R'_{3,2,1},\{2\}}(A) \, \Theta_{A}(R'_{3,2,1},\{2\};\cdot), 
\end{equation*}
where 
\begin{equation*}
Z_{R'_{3,2,0},\emptyset}(A) = A^{(3-4)(1-0)} \, \overline{C_{1}}(A) = A^{-1} \, C_{1}(A^{-1}) = A^{-1}+A^{3}+A^{7},
\end{equation*}
\begin{equation*}
Z_{R'_{3,2,1},\{1\}}(A) = A^{(3-4)(1-1)} \, \overline{C_{3}}(A) = C_{3}(A^{-1}) = A^{3},
\end{equation*}
and
\begin{equation*}
Z_{R'_{3,2,1},\{2\}}(A) = A^{(3-4)(1-1)} \, \overline{C_{2}}(A) = C_{2}(A^{-1}) = A+A^{5}.
\end{equation*}
Formulas for $\overline{C_{1}}(A)$, $\overline{C_{2}}(A)$, and $\overline{C_{3}}(A)$ above are obtained from $C_{1}(A)$, $C_{2}(A)$, and $C_{3}(A)$ by applying Remark~\ref{rem:prop_coef}(i).
\end{example}

\begin{lemma}
\label{lem:pf_thm_main}
Given $(R,I) \in \mathcal{W}$, assume that there is a finite subset $\mathcal{P}'$ of $\mathcal{W}$ with the following properties
\begin{enumerate}
\item[i)] $\mathcal{P}' \trianglerighteq (R,I)$,
\item[ii)] $(R',I')$ has a $\Theta_{A}$-state expansion for every $(R',I') \in \mathcal{P}'$, and
\item[iii)] there are rational functions $Q_{R',I'}(A)$, $(R',I') \in \mathcal{P}'$, such that
\begin{equation*}
\Theta_{A}(R,I;\cdot) = \sum_{(R',I') \in \mathcal{P}'} Q_{R',I'}(A) \, \Theta_{A}(R',I';\cdot).
\end{equation*}
\end{enumerate}
Then $(R,I)$ also has a $\Theta_{A}$-state expansion.
\end{lemma}

\begin{proof}
For each $(R',I') \in \mathcal{P}'$, let
\begin{equation*}
\Theta_{A}(R',I';\cdot) = \sum_{(\tilde{R}',\tilde{I}') \in \tilde{\mathcal{P}}'_{R',I'}} Q'_{(R',I'),(\tilde{R}',\tilde{I}')}(A) \, \Theta_{A}(\tilde{R}',\tilde{I}';\cdot)
\end{equation*}
be its $\Theta_{A}$-state expansion. Then clearly $\tilde{R}'$ in each pair
\begin{equation*}
(\tilde{R}',\tilde{I}')\in \tilde{\mathcal{S}}' = \bigcup_{(R',I') \in \mathcal{P}'} \tilde{\mathcal{P}}'_{R',I'}
\end{equation*}
is a middle state. Moreover, since $\tilde{\mathcal{P}}'_{R',I'} \trianglerighteq (R',I')$ and $\mathcal{P}' \trianglerighteq (R,I)$, by Lemma~\ref{lem:transitivity_triangle}, $\tilde{\mathcal{P}}'_{R',I'} \trianglerighteq (R,I)$ for every $(R',I') \in \mathcal{P}'$ and consequently, $\tilde{\mathcal{S}}' \trianglerighteq (R,I)$. Furthermore, for each $(\tilde{R}',\tilde{I}') \in \tilde{\mathcal{S}}'$ define
\begin{equation*}
\tilde{Q}'_{(R',I'),(\tilde{R}',\tilde{I}')}(A) = 
\begin{cases}
Q'_{(R',I'),(\tilde{R}',\tilde{I}')}(A), & \text{if} \ (\tilde{R}',\tilde{I}') \in \tilde{\mathcal{P}}'_{R',I'}, \\
0, & \text{otherwise}, 
\end{cases}
\end{equation*}
then
\begin{eqnarray*}
\Theta_{A}(R,I;\cdot) &=& \sum_{(R',I') \in \mathcal{P}'} \, \sum_{(\tilde{R}',\tilde{I}') \in \tilde{\mathcal{P}}'_{R',I'}} Q_{R',I'}(A) \, Q'_{(R',I'),(\tilde{R}',\tilde{I}')}(A) \, \Theta_{A}(\tilde{R}',\tilde{I}';\cdot) \\
&=& \sum_{(R',I') \in \mathcal{P}'} \, \sum_{(\tilde{R}',\tilde{I}') \in \tilde{\mathcal{S}}'} Q_{R',I'}(A) \, \tilde{Q}'_{(R',I'),(\tilde{R}',\tilde{I}')}(A) \, \Theta_{A}(\tilde{R}',\tilde{I}';\cdot) \\
&=& \sum_{(\tilde{R}',\tilde{I}') \in \tilde{\mathcal{S}}'} \tilde{Q}_{\tilde{R}',\tilde{I}'}(A) \, \Theta_{A}(\tilde{R}',\tilde{I}';\cdot),
\end{eqnarray*}
where
\begin{equation*}
\tilde{Q}_{\tilde{R}',\tilde{I}'}(A) = \sum_{(R',I') \in \mathcal{P}'} Q_{R',I'}(A) \, \tilde{Q}'_{(R',I'),(\tilde{R}',\tilde{I}')}(A)
\end{equation*}
are rational functions. Finally, if we let
\begin{equation*}
\mathcal{S}' = \{(\tilde{R}',\tilde{I}') \in \tilde{\mathcal{S}}' \mid \tilde{Q}_{\tilde{R}',\tilde{I}'}(A) \neq 0\}, 
\end{equation*}
then
\begin{equation*}
\Theta_{A}(R,I;\cdot) = \sum_{(R',I') \in \mathcal{S}'} \tilde{Q}_{R',I'}(A) \, \Theta_{A}(R',I';\cdot)
\end{equation*}
is a $\Theta_{A}$-state expansion for $(R,I)$.
\end{proof}

\begin{theorem}
\label{thm:main}
Every $(R,I) \in \mathcal{W}$ has a $\Theta_{A}$-state expansion.
\end{theorem}

\begin{proof}
Our proof is by induction on $n = n_{t}(R)$. For $n = 0,1$ or when $R$ is a middle state, the statement is obvious since we can take $\mathcal{P}' = \{(R,I)\}$ and $Q_{R,I}(A) = 1$ in \eqref{eqn:Theta_state_expansion}.

Assume that \eqref{eqn:Theta_state_expansion} exists for all roof states with $n_{t}(R) < n$ and arbitrary $I \in \mathrm{Fin}(\mathbb{N})$. We notice that, if there exists a $\Theta_{A}$-state expansion for $(R,I)$ then Proposition~\ref{prop:P_prime_concat_reflect_formula}(ii) yields a $\Theta_{A}$-state expansion for $(R *_{v} M,I)$ for any middle state $M$ with $n_{t}(M) = n_{b}(R)$. Thus, it suffices to consider cases when $R$ is a top state.

For a top state $R$ with $n_{t}(R) = n$, define
\begin{equation*}
q(R) = 
\begin{cases}
\min\{j,n-j\}-u + n-v, & \text{if} \ \mathcal{J}(R) \neq \emptyset, \\
0, & \text{otherwise}, 
\end{cases}
\end{equation*}
where $j = j(R) = \min \mathcal{J}(R)$, $u = u(R) = |\{i \in \varphi_{n}(\overline{R}^{*}) \mid i \leq j\}|$, and $v = v(R) = \min (\mathcal{J}(R) \cup \{n\} \setminus \{j\})$ (see Figure~\ref{fig:pf_thm_main}). Notice that $u \leq \min\{j,n-j\}$ and $v \leq n$, so $q(R) \geq 0$. Moreover, $q(R) = 0$ if and only if $R = L(0,n)$ or $R = R'_{n,j,0}$, where $R'_{n,j,0}$ is shown in Figure~\ref{fig:lem_roof_R_nk_2}. Clearly, $R = L(0,n)$ is a middle state, so $\mathcal{J}(R) = \emptyset$ and thus $q(R) = 0$. For $R = R'_{n,j,0}$, we see that $u = \min\{j,n-j\}$ and $v = n$, so $q(R) = 0$ as well. Conversely, if $q(R) = 0$ and $R$ is not a middle state, then $0 \leq \min\{j,n-j\}-u = v-n \leq 0$. Thus, it must be $\min\{j,n-j\} = u$ and $v = n$. Since $v = n$ if and only if $\mathcal{J}(R) = \{j\}$, and $\min\{j,n-j\} = u$ if and only if arcs joining top-boundary points with bottom-boundary points of $R$ are all either to the left or to the right of the $u$ parallel copies of $e_{j}$, it follows that $R = R'_{n,j,0}$. If $R$ is a top state which is also a middle state, then clearly $R =  L(0,n)$.

\begin{figure}[ht]
\centering
\includegraphics[scale=1]{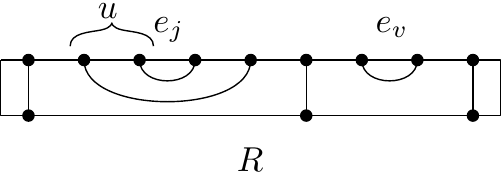}
\caption{$j = j(R)$, $u = u(R)$, and $v = v(R)$ for top state $R$ with $\mathcal{J}(R) \neq \emptyset$}
\label{fig:pf_thm_main}
\end{figure}

We use induction on $q = q(R)$ to show that for any top state $R$ with $n_{t}(R) = n$ and arbitrary $I \in \mathrm{Fin}(\mathbb{N})$, there is a $\Theta_{A}$-state expansion for $(R,I)$. 

If $q = 0$, then as we argued above $R = L(0,n)$ or $R = R'_{n,j,0}$ for some $j$. If $R = L(0,n)$, then $R$ is a middle state, so $(R,I)$ has a $\Theta_{A}$-state expansion. For $R = R'_{n,j,0}$, we see that by Lemma~\ref{lem:roof_R_nk_1} and Proposition~\ref{prop:P_prime_concat_reflect_formula}(i),
\begin{equation}
\label{eqn:pf_thm_main_1}
\Theta_{A}(R_{n,j},I;\cdot) = Z_{R'_{n,j,0},\emptyset}(A) \, \Theta_{A}(R,\emptyset \oplus I;\cdot) + \sum_{(R',I') \in \mathcal{P}'} Z_{R',I'}(A) \, \Theta_{A}(R',I' \oplus I;\cdot),
\end{equation}
where 
\begin{equation*}
\mathcal{P}' = \{(R'_{n,j,|J|},J) \mid J \in \mathcal{L}_{n}, \, 0 < |J| \leq \min\{j,n-j\}\} \trianglerighteq (R_{n,j},\emptyset)
\end{equation*} 
and coefficients $Z_{R',I'}(A)$ are given in Lemma~\ref{lem:roof_R_nk_2}. Let
\begin{equation*}
\tilde{\mathcal{P}}' = \{(R_{n,j},I)\} \cup \{(R',I' \oplus I) \mid (R',I') \in \mathcal{P}'\},
\end{equation*}
then by Proposition~\ref{prop:P_prime_concat_reflect_formula}(i), $\tilde{\mathcal{P}}' \trianglerighteq (R_{n,j},I)$. Since $n_{t}(R) = n_{t}(R_{n,j})$ and $n_{b}(R) = n_{b}(R_{n,j})$, it follows that $\tilde{\mathcal{P}}' \trianglerighteq (R,I)$. Moreover, each $(\tilde{R}',\tilde{I}') \in \tilde{\mathcal{P}}'$ has a $\Theta_{A}$-state expansion since either $\tilde{R}' = R_{n,j}$ is a middle state or $n_{t}(\tilde{R}') < n$ in which case induction can be applied to $(\tilde{R}',\tilde{I}')$. Furthermore, after solving \eqref{eqn:pf_thm_main_1} for $\Theta_{A}(R,I;\cdot)$,
\begin{equation*}
\Theta_{A}(R,I;\cdot) = \frac{1}{Z_{R'_{n,j,0},\emptyset}(A)} \, \Big( \Theta_{A}(R_{n,j},I;\cdot)-\sum_{(R',I') \in \mathcal{P}'} Z_{R',I'}(A) \, \Theta_{A}(R',I' \oplus I;\cdot) \Big).
\end{equation*}
Therefore, $\tilde{\mathcal{P}}'$ satisfies all assumptions of Lemma~\ref{lem:pf_thm_main}, consequently $(R,I)$ has a $\Theta_{A}$-state expansion.

Given $q > 0$, assume that for any top state $R$ with $n_{t}(R) = n$ and $q(R) < q$, and arbitrary $I \in \mathrm{Fin}(\mathbb{N})$, there is a $\Theta_{A}$-state expansion for $(R,I)$. For a top state $R$ with $n_{t}(R) = n$ and $q(R) = q$, let $\tilde{R} = \psi_{n}^{-1}(\{j\}) *_{v} R$ and let
\begin{equation*}
\mathcal{P}' = \{(\tilde{R}^{\tilde{J}},\tilde{I} \oplus I) \mid (\tilde{J},\tilde{I}) \in \mathcal{H}(\tilde{R}) \setminus (\{j\},\emptyset)\}.
\end{equation*} 

We show that every $(R',I') \in \mathcal{P}'$ has a $\Theta_{A}$-state expansion. If $(R',I') \in \mathcal{P}'$ with $n_{t}(R') = n$ then $R' = \tilde{R}^{J'}$ for some $J' = \{k\}$, where $k \in \mathcal{J}(\tilde{R}) \setminus \{j\}$, i.e., $k \geq v+1$. Since $j(R') = j+1$, $u(R') = u+1$, and $v(R') \geq v+1$, it follows that $q(R') < q$. Therefore, for each pair $(R',I') \in \mathcal{P}'$ either $n_{t}(R') < n$ or $n_{t}(R') = n$ and $q(R') < q$, hence either by induction on $n$ or on $q$, each $(R',I')$ has a $\Theta_{A}$-state expansion. In particular, we proved that $(R,I) \notin \mathcal{P}'$ and since $(R,I) = (\tilde{R}^{\{j\}},\emptyset \oplus I)$, by \eqref{eqn:1st_row_exp_over_P}, it follows that
\begin{equation}
\label{eqn:pf_thm_main_2}
\Theta_{A}(\tilde{R},I;\cdot) = A^{-n+2j} \, \Theta_{A}(R,I;\cdot) + \sum_{(R',I') \in \mathcal{P}'} Z_{R',I'}(A) \, \Theta_{A}(R',I';\cdot).
\end{equation}

Define $\delta: \mathbb{Z} \to \{0,1\}$ by $\delta(x) = 1$ if $x > 0$ and $\delta(x) = 0$ otherwise. We show that $(\tilde{R},I)$ has a $\Theta_{A}$-state expansion. Let $\tilde{R} = \tilde{R}' *_{v} \tilde{M}'$, where $\tilde{R}'$ is a top state and $\tilde{M}'$ is a middle state. Then clearly $j(\tilde{R}') = j$, $u(\tilde{R}') = u+\delta(\min\{j,n-j\}-u)$, and $v(\tilde{R}') = v+\delta(n-v)$. One checks that $u(\tilde{R}') + v(\tilde{R}') > u + v$ and consequently $q(\tilde{R}') < q$. By induction $(\tilde{R}',I)$ has a $\Theta_{A}$-state expansion, so Proposition~\ref{prop:P_prime_concat_reflect_formula}(ii) implies that there is a $\Theta_{A}$-state expansion for $(\tilde{R},I)$.

Let $\tilde{\mathcal{P}}' = \{(\tilde{R},I)\} \cup \mathcal{P}'$. Since $\mathcal{P}' \trianglerighteq (\tilde{R},I)$ by Proposition~\ref{prop:1st_row_exp}, $n_{t}(R) = n_{t}(\tilde{R})$, and $n_{b}(R) = n_{b}(\tilde{R})$, it follows that $\tilde{\mathcal{P}}' \trianglerighteq (R,I)$. Moreover, as we showed above, each $(\tilde{R}',\tilde{I}') \in \tilde{\mathcal{P}}'$ has a $\Theta_{A}$-state expansion and, using \eqref{eqn:pf_thm_main_2},
\begin{equation*}
\Theta_{A}(R,I;\cdot) = A^{n-2j} \Big( \Theta_{A}(\tilde{R},I;\cdot) - \sum_{(R',I') \in \mathcal{P}'} Z_{R',I'}(A) \, \Theta_{A}(R',I';\cdot) \Big).
\end{equation*}
Therefore, $\tilde{\mathcal{P}}'$ satisfies all assumptions of Lemma~\ref{lem:pf_thm_main} and consequently $(R,I)$ has a $\Theta_{A}$-state expansion.
\end{proof}

The proof of Theorem~\ref{thm:main} yields an algorithm for finding $\Theta_{A}$-state expansion for $(R,I)$.

\begin{algorithm}
\caption{$\Theta_{A}$-state expansion for $(R,I) \in \mathcal{W}$}
\label{alg:Theta_state_expansion}
\begin{algorithmic}[1]
\Procedure{Theta}{$R,I$}
\State {$n \gets n_{t}(R)$}
\If {\textbf{not} $\#(R\cap l_{i}^h)\leq n$ for all $i$} 
    \State {\Return $0$}
\ElsIf {$R$ is a middle state}
    \State {\Return $\Theta_{A}(R,I;\cdot)$}
\ElsIf {$\mathrm{ht}(R) > 0$} 
    \State {$(R,M) \gets R/l^{h}_{0}$}
    \State {$\sum_{i} Q_{i}(A) \, \Theta_{A}(R_{i},I_{i};\cdot) \gets \textsc{Theta}(R,I)$}
    \State {\Return $\sum_{i} Q_{i}(A) \, \Theta_{A}(R_{i} *_{v} M,I_{i};\cdot)$} 
\EndIf  
\State {$j \gets \min\mathcal{J}(R)$} 
\If {$R = R'_{n,j,0}$} 
    \State {$\mathrm{ExpSum} \gets \Theta_{A}(R_{n,j},I;\cdot)$}
    \For {$I' \in \mathcal{L}_{n}$ with $0 < |I'| \leq \min\{j,n-j\}$} 
        \State {$Z(A) \gets Z_{R'_{n,j,|I'|},I'}(A)$ from Lemma~\ref{lem:roof_R_nk_2}} 
        \State {$\mathrm{ExpSum} \gets \mathrm{ExpSum} - Z(A) \cdot \textsc{Theta}(R'_{n,j,|I'|},I' \oplus I)$}
    \EndFor
    \State {$Z(A) \gets Z_{R'_{n,j,0},\emptyset}(A)$ from Lemma~\ref{lem:roof_R_nk_2}} 
    \State {\Return $(Z(A))^{-1} \cdot \mathrm{ExpSum}$} 
\Else 
    \State {$R \gets \psi_{n}^{-1}(\{j\}) *_{v} R$}
    \State {$\mathrm{ExpSum} \gets \textsc{Theta}(R,I)$}
    \For {$(J',I') \in \mathcal{H}(R) \setminus \{(\{j\},\emptyset)\}$} 
        \State {$\mathrm{ExpSum} \gets \mathrm{ExpSum} - A^{-n+2\Vert{J'}\Vert-2\Vert{I'}\Vert} \cdot \textsc{Theta}(R^{J'},I' \oplus I)$}
    \EndFor
    \State {\Return $A^{n-2j} \cdot \mathrm{ExpSum}$}
\EndIf
\EndProcedure
\end{algorithmic}
\end{algorithm}

\begin{remark}
\label{rem:h_line_condi}
In line~3 of Algorithm~\ref{alg:Theta_state_expansion}, we check if a roof state $R$ satisfies the \emph{horizontal-line conditions}, i.e., $\#(R \cap l_{i}^h)\leq n_{t}(R)$ for all $i = 1,2,\ldots,\mathrm{ht}(R)$. This allows us to deal with trivial cases more efficiently. Indeed, by Theorem~\ref{thm:vh_line_condi}, if $R$ does not satisfy this condition, then $\Theta_{A}(R,I;F) = 0$ for any floor state $F$. In particular, if $n_{b}(R) > n_{t}(R)$ then $\Theta_{A}(R,I;\cdot) \equiv 0$.
\end{remark}

\begin{example}
\label{ex:Theta_state_expansion_2}
Let $R$ be the roof state in Figure~\ref{fig:ex_Theta_state_expansion_2}. We apply Algorithm~\ref{alg:Theta_state_expansion} to find a $\Theta_{A}$-state expansion for $(R,\emptyset)$ first and then use it to compute $C(A)$ for the Catalan state $C = R *_{v} F$, where $C$ and $F$ are shown in Figure~\ref{fig:ex_Theta_state_expansion_2}. Let $M = (M_{2,1})^{*}$ be a $\pi$-rotation of the middle state $M_{2,1}$ shown in Figure~\ref{fig:pf_lem_roof_R_nk_2}(a). Using \eqref{eqn:1st_row_exp} for $(R'_{4,1,0} *_{v} M,\emptyset)$, where top state $R'_{n,k,l}$ is shown in Figure~\ref{fig:lem_roof_R_nk_2}, we see that
\begin{equation*}
\Theta_{A}(R,\emptyset;\cdot) = A^{2} \, \Theta_{A}(R'_{4,1,0} *_{v} M,\emptyset;\cdot) - A^{6} \, \Theta_{A}(R'_{4,2,0},\emptyset;\cdot) - A^{4} \, \Theta_{A}(R'_{2,1,0},\{2\};\cdot) - A^{2} \, \Theta_{A}(R'_{2,1,0},\{3\};\cdot),
\end{equation*}
For a positive integer $n$, let $[n]_{q} =1+q+\cdots+q^{n-1}$ and let $[n] = [n]_{q = A^{4}}$. By Lemma~\ref{lem:roof_R_nk_1}, Lemma~\ref{lem:roof_R_nk_2}, and Proposition~\ref{prop:P_prime_concat_reflect_formula}(ii),
\begin{eqnarray*}
\Theta_{A}(R'_{4,1,0} *_{v} M,\emptyset;\cdot) &=& \frac{A^{14}}{[4]} \, \Theta_{A}(R_{4,1} *_{v} M,\emptyset;\cdot) - \frac{A^{2}[3]}{[4]} \, \Theta_{A}(R'_{4,1,1} *_{v} M,\{1\};\cdot) \\ 
&-& \frac{A^{4}[2]}{[4]} \, \Theta_{A}(R'_{4,1,1} *_{v} M,\{2\};\cdot) - \frac{A^{6}}{[4]} \, \Theta_{A}(R'_{4,1,1} *_{v} M,\{3\};\cdot)
\end{eqnarray*}
and
\begin{eqnarray*}
\Theta_{A}(R'_{4,2,0},\emptyset;\cdot) &=& \frac{A^{8}[2]}{[4][3]} \, \Theta_{A}(R_{4,2},\emptyset;\cdot) - \frac{A^{4}[2]}{[4]} \, \Theta_{A}(R'_{4,2,1},\{1\};\cdot) - \frac{A^{2}[2]^{2}}{[4]} \, \Theta_{A}(R'_{4,2,1},\{2\};\cdot) \\
&-& \frac{A^{4}[2]}{[4]} \, \Theta_{A}(R'_{4,2,1},\{3\};\cdot) - \frac{A^{8}[2]}{[4][3]} \, \Theta_{A}(R'_{4,2,2},\{1,2\};\cdot) - \frac{A^{6}[2]^{2}}{[4][3]} \, \Theta_{A}(R'_{4,2,2},\{1,3\};\cdot),
\end{eqnarray*}
where roof states $R_{n,k}$ are shown in Figure~\ref{fig:lem_roof_R_nk_1}.
Moreover, by Lemma~\ref{lem:roof_R_nk_1}, Lemma~\ref{lem:roof_R_nk_2}, and Proposition~\ref{prop:P_prime_concat_reflect_formula}(i),
\begin{equation*}
\Theta_{A}(R'_{2,1,0},I;\cdot) = \frac{A^{2}}{[2]} \, \Theta_{A}(R_{2,1},I;\cdot) - \frac{A^{2}}{[2]} \, \Theta_{A}(R'_{2,1,1},\{1\} \oplus I;\cdot)
\end{equation*}
for $I = \{1\},\{2\},\{3\}$. Since $R'_{4,1,1} *_{v} M = M$, $R'_{4,2,1} = R'_{2,1,0}$, and $R'_{4,2,2} = R'_{2,1,1} = R_{0,0}$, it follows that
\begin{eqnarray*}
\Theta_{A}(R,\emptyset;\cdot) 
&=& \frac{A^{16}}{[4]} \, \Theta_{A}(R_{4,1} *_{v} M,\emptyset;\cdot) - \frac{A^{4}[3]}{[4]} \, \Theta_{A}(M,\{1\};\cdot) - \frac{A^{6}[2]}{[4]} \, \Theta_{A}(M,\{2\};\cdot) - \frac{A^{8}}{[4]} \, \Theta_{A}(M,\{3\};\cdot)\\ 
&-&  \frac{A^{14}[2]}{[4][3]} \, \Theta_{A}(R_{4,2},\emptyset;\cdot) + \frac{A^{12}}{[4]} \, \Theta_{A}(R_{2,1},\{1\};\cdot) + \frac{A^{6}(A^{4}-1)}{[4]} \, \Theta_{A}(R_{2,1},\{2\};\cdot) \\ 
&-& \frac{A^{4}}{[4]} \, \Theta_{A}(R_{2,1},\{3\};\cdot) + \frac{A^{6}}{[3][2]} \, \Theta_{A}(R_{0,0},\{1,2\};\cdot) + \frac{A^{4}}{[3]} \, \Theta_{A}(R_{0,0},\{1,3\};\cdot).
\end{eqnarray*}

We see that $F_{\{2\}} = F_{\{1,2\}} = K_{0}$ and
\begin{equation*}
\Theta_{A}(R_{4,1} *_{v} M,\emptyset;F) = A^{-24} \, [2]^{3} [3]^{2}, \quad
\Theta_{A}(R_{4,2},\emptyset;F) = A^{-14} \, [2]^{2} [3]^{2}, 
\end{equation*}
\begin{equation*}
\Theta_{A}(R_{2,1},\{1\};F) = \Theta_{A}(M,\{1\};F) = \Theta_{A}(R_{2,1},\{3\};F) = \Theta_{A}(M,\{3\};F) = A^{-4} \, [2], 
\end{equation*}
by Theorem~\ref{thm:coef_no_bot_rtn} and $\Theta_{A}(R_{0,0},\{1,3\};F) = 0$ by Theorem~\ref{thm:vh_line_condi},
hence
\begin{equation*}
C(A) = \Theta_{A}(R,\emptyset;F) = \frac{A^{-8} \, [2]^{2}(1 + 3A^{4} + 2A^{8} + 3A^{12} + A^{16})}{[4]} = A^{-8}(1+A^{4})(1+3A^{4}+A^{8}).
\end{equation*}
\end{example}

\begin{figure}[ht]
\centering
\includegraphics[scale=1]{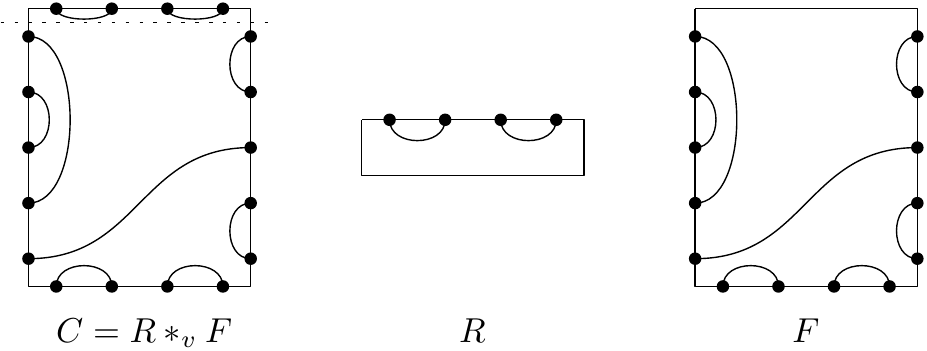}
\caption{Catalan state $C$, roof state $R$, and floor state $F$}
\label{fig:ex_Theta_state_expansion_2}
\end{figure}

Given $1 \leq j \leq n-1$, we let $1 \leq u \leq \min\{j,n-j\}$ and denote by
\begin{equation*}
\mathcal{C}_{n,u} = \{ (I_{1},I_{2},\ldots,I_{t}) \mid \emptyset = I_{1} \prec I_{2} \prec \cdots \prec I_{t} \in \mathcal{L}_{n}, \, |I_{t}| \leq u, \, t \geq 1\}.
\end{equation*}
For $c = (I_{1},\ldots,I_{t}) \in \mathcal{C}_{n,u}$ let $I_{c} = I_{t}$ and define
\begin{equation}
\label{eqn:pi_nj}
\pi_{n,j,c}(A) =  \frac{1}{Z_{n-2|I_{c}|,j-|I_{c}|,\emptyset}(A)} \, \prod_{i=1}^{t-1} \Big( -\frac{Z_{n-2|I_{i}|,j-|I_{i}|,I_{i+1} \ominus I_{i} }(A)}{Z_{n-2|I_{i}|,j-|I_{i}|,\emptyset}(A)} \Big),
\end{equation}
where $Z_{n,j,I}(A) = Z_{R'_{n,j,|I|},I}(A)$ is as in Lemma~\ref{lem:roof_R_nk_2}.

\begin{corollary}
\label{cor:formula_T_nju}
Let $T_{n,j,u}$ be the top state shown in Figure~\ref{fig:cor_formula_T_nju} with $u \geq 1$. Then 
\begin{equation}
\label{eqn:cor_formula_T_nju}
\Theta_{A}(T_{n,j,u},\emptyset;\cdot) = A^{(n-2j)d} \sum_{c \in \mathcal{C}_{n,u}} \pi_{n,j,c}(A) \, \Theta_{A}(R_{n-2|I_{c}|,j-|I_{c}|} *_{v} M_{n-2u,d},I_{c};\cdot),
\end{equation}
where $d = \min\{j,n-j\}-u$, $R_{n,k}$ is shown in Figure~\ref{fig:lem_roof_R_nk_1}, $M_{n,k}$ is shown in Figure~\ref{fig:pf_lem_roof_R_nk_2}(a), and $\pi_{n,j,c}(A)$ is given by \eqref{eqn:pi_nj}.
\end{corollary}

\begin{figure}[ht]
\centering
\includegraphics[scale=1]{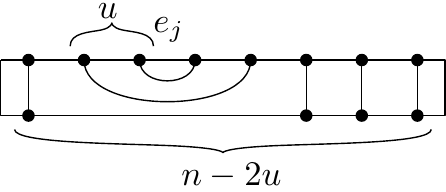}
\caption{Top state $T_{n,j,u}$}
\label{fig:cor_formula_T_nju}
\end{figure}

\begin{proof}
Let $\tilde{R} = R'_{n,j,0} *_{v} M_{n-2u,d}$, where $R'_{n,j,0}$ is as in Figure~\ref{fig:lem_roof_R_nk_2} and denote by $Z_{n,j,I}(A)$ the Laurent polynomial $Z_{R'_{n,j,|I|},I}(A)$ given in Lemma~\ref{lem:roof_R_nk_2}. 
Using Lemma~\ref{lem:roof_R_nk_1} and Proposition~\ref{prop:P_prime_concat_reflect_formula}(ii),
\begin{equation}
\Theta_{A}(\tilde{R},\emptyset;\cdot) = \frac{1}{Z_{n,j,\emptyset}(A)} \Big( \Theta_{A}(R_{n,j} *_{v} M_{n-2u,d},\emptyset;\cdot) - \sum_{(R',I') \in \mathcal{P}'} Z_{n,j,I'}(A) \, \Theta_{A}(R' *_{v} M_{n-2u,d},I';\cdot) \Big)
\label{eqn:pf_cor_formula_T_nju_1}
\end{equation}
where $\mathcal{P}' = \{(R'_{n,j,|I'|},I') \mid I' \in \mathcal{L}_{n}, \, 0 < |I'| \leq \min\{j,n-j\}\}$.
Since $R'_{n,j,|I'|} = R'_{n-2|I'|,j-|I'|,0}$, by Lemma~\ref{lem:roof_R_nk_1} and Proposition~\ref{prop:P_prime_concat_reflect_formula}(i) and (ii), each $\Theta_{A}(R' *_{v} M_{n-2u,d},I';\cdot)$ on the right-hand side of \eqref{eqn:pf_cor_formula_T_nju_1} can be written as
\begin{eqnarray}
\label{eqn:pf_cor_formula_T_nju_2}
\Theta_{A}(R' *_{v} M_{n-2u,d},I';\cdot) &=& \frac{1}{Z_{n-2|I'|,j-|I'|,\emptyset}(A)} \Big( \Theta_{A}(R_{n-2|I'|,j-|I'|} *_{v} M_{n-2u,d},\emptyset \oplus I';\cdot) \\
&-& \sum_{(R'',I'') \in \mathcal{P}''_{R',I'}} Z_{n-2|I'|,j-|I'|,I''}(A) \, \Theta_{A}(R'' *_{v} M_{n-2u,d},I'' \oplus I';\cdot) \Big) \notag,
\end{eqnarray}
where $\mathcal{P}''_{R',I'}= \{(R'_{n-2|I'|,j-|I'|,|I''|},I'') \mid I'' \in \mathcal{L}_{n-2|I'|}, \, 0 < |I''| \leq \min\{j,n-j\}-|I'|\}$.
Substituting \eqref{eqn:pf_cor_formula_T_nju_2} into \eqref{eqn:pf_cor_formula_T_nju_1} yields
\begin{eqnarray*}
\Theta_{A}(\tilde{R},\emptyset;\cdot) &=& \frac{1}{Z_{n,j,\emptyset}(A)} \, \Theta_{A}(R_{n,j} *_{v} M_{n-2u,d},\emptyset;\cdot) \\
&-& \sum_{(R',I') \in \mathcal{P}'} \frac{Z_{n,j,I'}(A)}{Z_{n,j,\emptyset}(A)} \, \frac{1}{Z_{n-2|I'|,j-|I'|,\emptyset}(A)} \, \Theta_{A}(R_{n-2|I'|,j-|I'|} *_{v} M_{n-2u,d},I';\cdot) \\
&+& \sum_{(R',I') \in \mathcal{P}'} \, \sum_{(R'',I'') \in \mathcal{P}''_{R',I'}} \frac{Z_{n,j,I'}(A)}{Z_{n,j,\emptyset}(A)} \, \frac{Z_{n-2|I'|,j-|I'|,I''}(A)}{Z_{n-2|I'|,j-|I'|,\emptyset}(A)} \, \Theta_{A}(R'' *_{v} M_{n-2u,d},I'' \oplus I';\cdot).
\end{eqnarray*}
After recursively applying Lemma~\ref{lem:roof_R_nk_1} and Proposition~\ref{prop:P_prime_concat_reflect_formula}(i) and (ii), we see that
\begin{equation}
\label{eqn:pf_cor_formula_T_nju_3}
\Theta_{A}(\tilde{R},\emptyset;\cdot) = \sum_{c \in \mathcal{C}_{n,u+d}} \pi_{n,j,c}(A) \, \Theta_{A}(R_{n-2|I_{c}|,j-|I_{c}|} *_{v} M_{n-2u,d},I_{c};\cdot).
\end{equation}

Applying $d$ times the first-row expansion \eqref{eqn:1st_row_exp} to $(\tilde{R},\emptyset)$ gives
\begin{equation}
\Theta_{A}(\tilde{R},\emptyset;\cdot) = A^{(-n+2j)d} \, \Theta_{A}(T_{n,j,u},\emptyset;\cdot).
\label{eqn:pf_cor_formula_T_nju_4}
\end{equation}
Since all terms $\Theta_{A}(R_{n-2|I_{c}|,j-|I_{c}|} *_{v} M_{n-2u,d},I_{c};\cdot)$ with $|I_{c}| > u$ in the sum are zero by Remark~\ref{rem:h_line_condi}, we see that \eqref{eqn:cor_formula_T_nju} follows from \eqref{eqn:pf_cor_formula_T_nju_3} and \eqref{eqn:pf_cor_formula_T_nju_4}.
\end{proof}

Note that \eqref{eqn:cor_formula_T_nju} after collecting like-terms can be used to obtain $\Theta_{A}$-state expansions for the family of top states $T_{n,j,u}$. Consequently, by Proposition~\ref{prop:P_prime_concat_reflect_formula}(ii), we can obtain $\Theta_{A}$-state expansions for infinitely many roof states $R = T_{n,j,u} *_{v} M$, where $M$ is a middle state with $n_{t}(M) = n-2u$.

We say that a coefficient $C(A)$ of a Catalan state $C$ is \emph{unimodal} if there is an integer $k$ and a unimodal polynomial\footnote{A polynomial $p(x) = \sum_{i=0}^{n} a_{i} x^{i}$ is unimodal if its coefficients $a_{0},a_{1},\ldots,a_{n}$ form a unimodal sequence.} $p(x)$ such that $C(A) = A^{k} \, p(A^{4})$. Clearly, coefficients of Catalan states of $L(m,n)$ for $n = 1,2$ are unimodal and, using Proposition~5.4 of \cite{DP2019}, it can also easily be seen that coefficients of Catalan states of $L(m,3)$ are unimodal. Furthermore, preliminary results that we obtained using $\Theta_{A}$-state expansion for Catalan states of $L(m,4)$ indicate that quite likely coefficients for these states are unimodal as well. However, this is not the case for coefficients of Catalan states of $L(m,5)$ as the following example shows.

\begin{figure}[ht]
\centering
\includegraphics[scale=1]{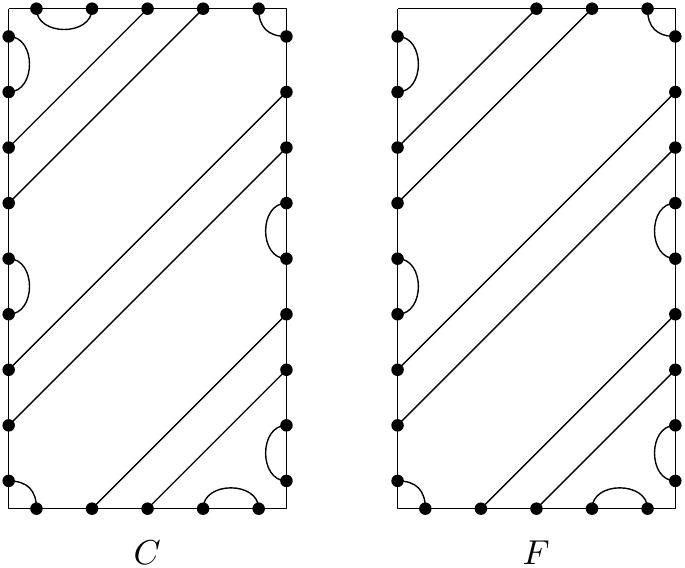}
\caption{Catalan state $C$ and floor state $F$}
\label{fig:ex_Theta_state_expansion_3}
\end{figure}

\begin{example}
\label{ex:Theta_state_expansion_3}
Let $C$ be the Catalan state in Figure~\ref{fig:ex_Theta_state_expansion_3}, then $C = R'_{5,1,0} *_{v} F$ for the top state $R'_{n,k,l}$ in Figure~\ref{fig:ex_Theta_state_expansion_2} and the floor state $F$ shown in Figure~\ref{fig:ex_Theta_state_expansion_3}. Since $F_{I} = K_{0}$ for all $I \in \mathcal{L}_{5} \setminus \{\emptyset,\{4\}\}$, by Lemma~\ref{lem:roof_R_nk_1}, one shows that
\begin{equation*}
C(A) = \Theta_{A}(R'_{5,1,0},\emptyset;F) = \frac{1}{Z_{R'_{5,1,0},\emptyset}(A)} \, \Theta_{A}(R_{5,1},\emptyset;F) - \frac{Z_{R'_{5,1,1},\{4\}}(A)}{Z_{R'_{5,1,0},\emptyset}(A)} \, \Theta_{A}(R'_{5,1,1},\{4\};F),
\end{equation*}
where $Z_{R'_{5,1,0},\emptyset}(A) = A^{-23}(1+A^{4}+A^{8}+A^{12}+A^{16})$ and $Z_{R'_{5,1,1},\{4\}}(A) = A^{-15}$ are given in Lemma~\ref{lem:roof_R_nk_2}. Using Theorem~\ref{thm:coef_no_bot_rtn} one verifies that 
\begin{equation*}
\Theta_{A}(R_{5,1},\emptyset;F) = A^{-24}(A^{-12} + A^{-8} + A^{-4} + 2 + A^{4})^{3}
\end{equation*} 
and 
\begin{equation*}
\Theta_{A}(R'_{5,1,1},\{4\};F) = A^{-9}.
\end{equation*}
Therefore,
\begin{equation*}
C(A) = A^{-37}(1 + 2A^{4} + 3A^{8} + 7A^{12} + 8A^{16} + 7A^{20} + 9A^{24} + 5A^{28} + A^{32}).
\end{equation*} 
\end{example}

\section*{Acknowledgement} Authors would like to thank Professor J\'{o}zef H. Przytycki for all valuable comments and suggestions.

\bibliography{main}

\end{document}